\documentclass[reqno]{amsart}
\usepackage{hyperref,xcolor,amssymb}
\usepackage[margin=1.3in]{geometry}

\newcommand{\dd}{\mathrm{d}}

\def\Xint#1{\mathchoice
	{\XXint\displaystyle\textstyle{#1}}%
	{\XXint\textstyle\scriptstyle{#1}}%
	{\XXint\scriptstyle\scriptscriptstyle{#1}}%
	{\XXint\scriptscriptstyle\scriptscriptstyle{#1}}%
	\!\int}
\def\XXint#1#2#3{{\setbox0=\hbox{$#1{#2#3}{\int}$ }
		\vcenter{\hbox{$#2#3$ }}\kern-.6\wd0}}

\def\dashint{\Xint-}

\begin{document}
\title[Higher integrability and stability of $(p,q)$-quasiminimizers]
{Higher integrability and stability of $(p,q)$-quasiminimizers}

\author[A. Nastasi, C. Pacchiano Camacho]
{Antonella Nastasi, Cintia Pacchiano Camacho}

\address{Antonella Nastasi \newline
University of Palermo, Department of Mathematics and Computer Science, Via Archirafi 34, 90123, Palermo, Italy}
\email{antonella.nastasi@unipa.it, antonellanastasi.math@gmail.com}

\address{Cintia Pacchiano Camacho\newline
	Aalto University, Department of Mathematics and Systems Analysis, Espoo, Finland}
\email{cintia.pacchiano@aalto.fi}


\begin{abstract}
Using purely variational methods, we prove local and global higher integrability results for upper gradients of quasiminimizers of a $(p,q)$-Dirichlet integral with fixed boundary data, assuming it belongs to a slightly better Newtonian space.  We also obtain a stability property with respect to the varying exponents $p$ and $q$. The setting is a doubling metric measure space supporting a Poincar\'e inequality. 
\end{abstract}

\maketitle
\numberwithin{equation}{section}
\newtheorem{theorem}{Theorem}[section]
\newtheorem{lemma}[theorem]{Lemma}
\newtheorem{proposition}[theorem]{Proposition}
\newtheorem{remark}[theorem]{Remark}
\newtheorem{definition}[theorem]{Definition}
\newtheorem{corollary}[theorem]{Corollary}
\allowdisplaybreaks

\section{Introduction}
We study regularity for the gradient of quasiminimizers to the following $(p,q)$-Dirichlet integral 
\begin{equation}\label{J}
	\int_{\Omega}  (a g_{u}^p + b g_{u}^q) \,\dd \mu,
\end{equation}
in the context of metric measure spaces, where $1<p<q$ and $g_u$ is the minimal $q$-weak upper gradient of $u$. Here, $(X,d,\mu)$ is a complete metric measure space endowed with a metric $d$ and a doubling measure $\mu$, supporting a weak $(1,p)$-Poincar\'e inequality and $\Omega\subset X$ is an open bounded set, whose complement satisfies a uniform $p$-fatness condition. Furthermore, we consider some coefficient functions $a$ and $b$ to be measurable and satisfying $0\leq\alpha\leq a,b \leq\beta$, for some positive constants $\alpha$ and $\beta$. This study extends the work in \cite{MZG} since this only concerns $p$-quasiminimizers. We give not only qualitative properties but also quantitative ones providing an explicit analysis on the dependencies of
the constants.

Quasiminimizers were originally introduced by Giaquinta and Giusti \cite{GG1,GG2} as a unifying tool. The main advantage is that the definition of quasiminimizers in $\mathbb{R}^n$ depends only on moduli of gradients, therefore we can substitute them by upper gradients. This way, the theory of partial differential equations is generalized to metric measure spaces. Quasiminimizers have been an active research topic for several years in the setting of doubling metric measure spaces supporting a Poincar\'e inequality. For instance, the boundary continuity for $(p,q)$-quasiminimizers on a bounded set $\Omega$ with fixed boundary data has been examined in \cite{NP}, furthermore, it was proven that $(p,q)$-quasiminimizers are (locally) H\"older continuous. In \cite{KMM}, Kinnunen, Marola, and Martio proved that an increasing sequence of quasiminimizers converges locally uniformly to a quasiminimizer, provided that the limit function is finite at some point.

In this manuscript, we first show a global higher integrability for upper gradients of $(p,q)$-quasiminimizers of the Dirichlet integral \eqref{J} with fixed boundary data, see Theorem \ref{Theorem 5.2}. In the Euclidean setting, the first results concerning local higher integrability are by Bojarski \cite{Boj} and Elcrat and Meyers \cite{EM}. For more local results, see also \cite{CW, GM, G, Z}. In \cite{Gra}, Granlund showed that if the complement of the domain satisfies a certain measure density condition, then minimizers have a global higher integrability property. Later, Kilpel\"ainen and Koskela \cite{KilKos} generalized this result to a uniform capacity density condition. In order to prove global higher integrability, we also require a regularity condition for the complement of the domain, more specifically, $X\setminus\Omega$ is assumed to be uniformly $p$-fat. The main steps in our proof are showing that minimal upper gradients satisfy a reverse type H\"older inequality, applying Gehring's lemma \cite{Maa, ZG} and finally, generalizing the resulting local higher integrability to the entire $\Omega$. Therefore, an appropriate covering argument is needed. Since we are considering quasiminimizers with fixed boundary data, we are able to work near and also on the boundary.  Consequently, we can cover $\Omega$ by balls that are inside the set, together with those that intersect the complement. On the one hand, when working inside $\Omega$, a De Giorgi type inequality \eqref{5.4} implies that the minimal upper gradient satisfies a reverse H\"older inequality. On the other hand, we have to be careful when working up to the boundary and use some delicate extra tools. Here the $p$-fatness of the complement plays an important role. Furthermore, we need self-improving properties of the Poincar\'e inequality (see \cite{KZ}) and the $p$-fatness condition (see \cite{BMS}). These results allow us to use a capacity version of a Sobolev-Poincar\'e type inequality, i.e. a Maz'ya type estimate (see \cite{B}) and so, obtain the desired reverse H\"older inequality. Finally, since $\Omega$ is by hypothesis a bounded set, it is possible to obtain a finite covering using balls that either are contained in $\Omega$ or  intersect its complement. There is a rich literature concerning higher integrability results in the Euclidean  setting regarding both elliptic and parabolic cases, see for example \cite{BDKS, CM, ELM, GS, LM1, MSSS}. In particular, Colombo and Mingione \cite{CM} prove local regularity of the gradient for minimizers for double phase variational problems. Instead, for the metric setting we refer the reader to \cite{MZG} for the elliptic case and to \cite{FH, MMM, MM} for the parabolic case.

The other result obtained in our study is a stability property, see Theorem \ref{Theorem 1.1.}. We consider a sequence $(u_i)$ of $(p_i,q_i)$-quasiminimizers of the corresponding integral \eqref{J}. We assume that all functions $u_i$ have the same boundary data and quasiminimizing constant. Unlike the Euclidean case studied in \cite{KLM}, where the authors prove that minimizers with varying exponent converge to the solution of the limit problem in the Sobolev space, we need to assume a priori that the sequence of $(p_i,q_i)$-quasiminimizers converges to a certain $u$. Roughly speaking, we need this convergence assumption from the beginning because when working with quasiminimizers, instead of minimizers, we loose uniqueness and thus, the uniqueness of the limit. We prove that if the sequences $(p_i)$ and $(q_i)$ converge respectively to $p$ and $q$, then there exists a $(p,q)$-quasiminimizer $u$ of \eqref{J} with the same boundary data and, furthermore, the sequence $(u_i)$ converges to $u$. To be able to prove this convergence, we use the global higher integrability. Li and Martio \cite{LM} examined a quasilinear elliptic operator and proved a corresponding convergence result for solutions of an obstacle problem in a bounded subset of $\mathbb{R}^n$. Later, they proved a similar result for a double obstacle problem \cite{LM1}. For more references concerning stability results in the Euclidean setting, we refer the reader to \cite{KP, LM, L2} and the references therein.

This paper is motivated by the work of Maasalo and Zatorska-Goldstein \cite{MZG} and is a continuation of \cite{NP}. The novelty is that, as already anticipated, we include both $p$-Laplace and $q$-Laplace operators, involving also some measurable coefficient functions $a$ and $b$, assuming only they are bounded away from zero and infinity. This condition over the coefficients is essential for our approach, specifically when using Maz'ya type estimate for the higher integrability up to the boundary part, however it is an open question if it could be relaxed. Already in the Euclidean setting, it would be interesting to find out if one can establish global higher integrability with assumptions as in \cite{CM}. To see more about non standard growth conditions, see \cite{DGV, Ma0, Ma1, Ma2, Ma3}.

The present work is organized as follows: in Section \ref{Sec2} we fix the general setup and we
present basic facts about analytic tools used in metric setting. The results are
stated without proofs. The reader familiar with metric measure spaces may omit this part. In particular, in Section \ref{extrasection}, we present Newtonian spaces, results related to the Poincar\'e inequality and also some useful properties concerning the space with zero boundary values. Section \ref{newsection3} is devoted to introduce the concept of $(p,q)$-quasiminimizers and report the De Giorgi type inequality, as proven in \cite{NP}. Section \ref{Sec3} deals with the higher integrability problem for quasiminimizers and Section \ref{Sec4} contains the proof of the stability result.

\section*{Acknowledgements}
The second author was supported by a doctoral training grant for 2021 from the V\"ais\"al\"a Fund.

\section{Mathematical background}\label{Sec2}
Let $(X, d, \mu)$ be a complete metric measure space, where $\mu$ is a Borel regular measure with $\mu(E)>0$ for every $E\subset X$ non empty open set and $\mu(A)<\infty$ for every $A\subset X$ bounded set. Let $B=B(y,r)\subset X$ be a ball with center $y \in X$ and radius $r>0$. When there is no possibility of confusion, we denote by $\lambda B$ a ball with the same centre as
$B$ but $\lambda$ times its radius.

We set $$u_S= \dfrac{1}{\mu(S)}\int_{S} u \, \dd\mu=\dashint_{S}u\, \dd\mu,$$where $S \subset X$ is a measurable set of finite positive measure and $u: S \to \mathbb{R}$ is a measurable function. Throughout this paper, we will indicate with $C$ all positive constants, even if they assume different values, unless otherwise specified. 

\begin{definition}[\cite{BB}, Section 3.1]
	A measure $\mu$ on $X$ is said to be doubling if there exists a constant $C_d \geq 1$, called the doubling constant, such that for every ball $B$ in $X$
	\begin{equation}\label{doubling}
		0<\mu(2B)\leq C_d \, \mu(B)< \infty.
	\end{equation} 
\end{definition}
\begin{lemma}[\cite{BB}, Lemma 3.3]\label{lemm3.3}
	Let $(X, d, \mu)$ be a metric measure space with $\mu$ doubling. Then there is $Q>0$ such that
	\begin{equation}\label{s}
		\dfrac{\mu(B(y,\rho))}{\mu(B(x, R))}\geq C\left(\dfrac{\rho}{R}\right)^Q,
	\end{equation}
	for all $\rho\in ]0, R]$, $x \in {\Omega}$, $y \in B(x, R)$, where constants $Q$ and $C$ depend only on $C_d$.
\end{lemma}

\begin{definition}[\cite{MZG}, Section 2.1.7] A metric space $X$ is said to be linearly locally connected, denoted as LLC, if there exist constants $C\geq 1$ and $r_0>0$ such that for all balls $B$ in $X$ with radius at most $r_0$, every pair of distinct points in the annulus $2B\setminus\bar{B}$ can be connected by a curve lying in the annulus $2B\setminus C^{-1}\bar{B}$.
\end{definition}

\subsection{Upper gradients}\label{Upper gradient} We introduce the notion of upper gradient as a way to overcome the lack of a differentiable structure in the metric setting. Upper gradients are a generalization of the modulus of the gradient in the Euclidean case. For further details, we refer the reader to the book by  Bj\"{o}rn and Bj\"{o}rn \cite{BB}.

\begin{definition}[\cite{BB}, Definition 1.13]
	A non negative Borel measurable function $g$ is said to be an upper gradient of function $u: X \to [-\infty, +\infty]$ if, for all compact rectifiable arc length parametrized paths $\gamma$ connecting $x$ and $y$, we have
	\begin{equation}\label{ug}
		|u(x)-u(y)|\leq \int_{\gamma}g\, \dd s,
	\end{equation}
	whenever $u(x)$ and $u(y)$ are both finite and $\int_{\gamma}g \, \dd s= \infty$ otherwise.
\end{definition}
Notice that, as a consequence of the last definition, if $g$ is an upper gradient of  $u$ and $\phi$ is any non negative Borel measurable function, then $g+\phi$ is still an upper gradient of $u$. To avoid this lack of uniqueness, we introduce $q$-weak upper gradients. More specifically, if $g$ satisfies \eqref{ug} for $q$-almost all paths, then $g$ is a $q$-weak upper gradient of $u$. The next result gives us the existence of a unique minimal $q$-weak upper gradient of $u$.
\begin{theorem}[\cite{BB}, Theorem 2.5]\label{mpwug}
	Let $q \in ]1, \infty[$. Suppose that $u\in L^q(X)$ has an $L^q(X)$ integrable $q$-weak upper gradient. Then there exists a $q$-weak upper gradient, denoted with $g_u$, such that $g_u\leq g$ $\mu$-a.e. in $X$, for each $q$-weak upper gradient $g$ of $u$. This $g_u$ is called the minimal $q$-weak upper gradient of $u$.
\end{theorem}

\subsection{Poincar\'{e} inequalities}\label{Poincare inequalities}
In general, the upper gradients of a function do not necessarily give us a control over it. In order to gain such control, one standard hypothesis when working in the metric setting is to assume that the space supports a Poincar\'{e} inequality.

\begin{definition}\label{PI}
	Let $p \in [1, \infty[$. 
	A metric measure space $X$ supports a weak $(1, p)$-Poincar\'{e} inequality if there exist $C_{PI}$ and a dilation factor $\lambda \geq 1 $ such that 
	\begin{equation*}
		\dashint_{B} |u-u_{B}|\, \dd\mu\leq C_{PI} r \left(\dashint_{\lambda B}g_u^p \, \dd\mu\right)^{\frac{1}{p}},
	\end{equation*}
	for all balls $B=B(y,r)\subset X$ and for all $u \in L^1_{loc}(X)$.
\end{definition}

	We note that a standard, yet non-trivial, assumption in the metric setting is that the space satisfies a weak $(1, s)$-Poincar\'{e} inequality for some $s<p$, where $p$ is the smallest natural exponent associated with the studied problem. However, as shown by Keith and Zhong \cite{KZ}, the Poincar\'{e} inequality is a self-improving property.
\begin{theorem}[\cite{KZ}, Theorem 1.0.1]\label{spoincare}
	Let $(X, d, \mu)$ be a complete metric measure space with $\mu$ Borel and doubling, supporting a weak $(1,p)$-Poincar\'{e} inequality for $p>1$, then there exists $\epsilon>0$ such that $X$ supports a weak $(1, s)$-Poincar\'{e} inequality for every $s>p-\epsilon$. Here, $\epsilon$ and the constants associated
with the $(1, s)$-Poincar\'e inequality depend only on $p$, the doubling
constant $C_d$ and $C_{PI}$.
\end{theorem}

The following results show some further self-improving properties of the weak $(1,s)$-Poincar\'{e} inequality.
\begin{theorem}[\cite{BB}, Theorem 4.21] 
\label{sstars}
	Assume that $X$ supports a weak $(1,s)$-Poincar\'{e} inequality and that $Q$ in (\ref{s}) satisfies $Q>s$. Then $X$ supports a weak $(s^*,s)$-Poincar\'{e} inequality with $s^*=\frac{Qs}{Q-s}$. More precisely, there are constants $C$ and a dilation factor $\lambda'>1$ such that
	\begin{equation}\label{(2.8)}
		\left(\dashint_{B} |u-u_{B}|^{s^*} \,\dd\mu\right)^{\frac{1}{s^{*}}}\leq C r \left(\dashint_{\lambda'B}g_u^s \, \dd\mu\right)^{\frac{1}{s}},
	\end{equation}
	for all balls $B=B(y,r)\subset X$ and all integrable functions $u$ in $B(y, r)$. The constant $C$ depends on $C_{PI}$ and the dilation factor $\lambda'$ depends on $\lambda$ from Definition \ref{PI}.
\end{theorem}

\begin{corollary}[\cite{BB}, Corollary 4.26]
\label{coropoinca}
If $X$ supports a weak $(1,s)$-Poincar\'{e} inequality and $Q$ in (\ref{s}) satisfies $Q\leq s$, then $X$ supports a weak $(t,s)$-Poincar\'{e} inequality for all $1\leq t<\infty$.
\end{corollary}

\begin{remark}\label{rem1poin} By the H\"older inequality we see that a weak $(s^*,s)$-Poincar\'{e} inequality implies the same inequality for smaller values of $s^*$. Meaning that $X$ will then support a weak $(t,s)$-Poincar\'{e} inequality for all $1<t<s^*$.
\end{remark}

\begin{remark}\label{rem2poin} The exponent $Q$ in (\ref{s}) is not uniquely determined, in particular, since $\rho< R$, we can always make $Q$ larger. Thus, the assumption $Q>s$ in Theorem \ref{sstars} can always be fulfilled. 
\end{remark}

\subsection{Newtonian spaces}\label{extrasection}
When working in the general metric setting we need to be careful about defining the suitable working space. That is why, before introducing the  $(p,q)$-Dirichlet boundary value problem,  we devote this section to collect some detailed properties and results concerning these function spaces. We define $\widetilde{N}^{1,q}(X)$ to be the space of all $q$-integrable functions $u$ on $X$ that have a $q$-integrable $q$-weak upper gradient $g$ on $X$. We equip this space with the seminorm $\Vert u\Vert_{\widetilde{N}^{1,q}(X)}=\Vert u\Vert_{L^q(X)}+\inf\Vert g\Vert_{L^q(X)}$, where the infimum is taken over all $q$-weak upper gradients of $u$. We define the equivalence relation in $\widetilde{N}^{1,q}(X)$ by saying that $u\sim v$ if $\Vert u-v\Vert_{\widetilde{N}^{1,q}(X)}=0$. The Newtonian space $N^{1,q}(X)$ is then defined by $\widetilde{N}^{1,q}(X)/\sim$, with the norm $\Vert u\Vert_{N^{1,q}(X)}=\Vert u\Vert_{\widetilde{N}^{1,q}(X)}$.

\begin{remark}[\cite{BB}, Corollary A.9]\label{gradients}
	Let  $1<p<q$, $u \in N^{1,q}(X)$. If $(X, d, \mu)$  is a complete doubling $(1,p)$-Poincar\'{e} space, then the minimal $p$-weak upper gradient and the minimal $q$-weak upper gradient of $u$ coincide $\mu$-a.e.
\end{remark}
In virtue of Remark \ref{gradients}, we consider $q$-weak upper gradients rather than $p$-weak upper gradients.

\subsection{Capacities}\label{Capacities}
\begin{definition}[\cite{BB}, Definition 1.24] Let $E\subset X$ be a Borel set. We define the $p$-capacity of $E$ as
	\begin{equation*}
		{\rm C}_{p}(E)=\inf_{u}\left(\int_{X} |u|^{p} \, \dd\mu + \inf_{u}\int_{X} g_{u}^{p} \, \dd\mu\right),
	\end{equation*}
	where the infimum is taken over all $u\in N^{1,p}(X)$.
\end{definition}	
\noindent We say that a property holds $p$-quasieverywhere ($p$-q.e.) if the set of points
for which it does not hold has $p$-capacity zero.

\begin{definition}[\cite{BB}, Definition 6.13] Let $B\subset X$ be a ball and $E\subset B$. We define the variational capacity
	\begin{equation*}
		{\rm cap}_{p}(E;2B)=\inf_{u}\int_{2B}g_{u}^{p} \, \dd\mu,
	\end{equation*}
	where the infimum is taken over all $u\in N^{1,p}(2B)$ such that $u\geq 1$ on $E$ and $u = 0$ on $X \setminus 2B$ $p$-q.e..
\end{definition}
The following lemma compares the capacities ${\rm cap}_p$ and $C_p$ and shows that they are in many cases equivalent.

\begin{lemma}[\cite{BMS}, Lemma 2.6] Let $B$ be a ball in $X$ with radius $r$ and $E\subset B$ be a Borel set. Then for each $\lambda>1$ with $\lambda r<\frac{1}{3}{\rm diam} X$, there exists $C_{\lambda}>0$ such that
$$
\frac{\mu(E)}{C_{\lambda}r^{p}}\leq{\rm cap}_p(E;\lambda B)\leq\frac{C_\lambda\mu(B)}{r^p}
$$
and
$$
\frac{C_p(E)}{C_{\lambda}(1+r^{p})}\leq{\rm cap}_p(E;\lambda B)\leq C_\lambda\left(1+\frac{1}{r^p}\right)C_p(E).
$$
In particular, $C_p(E)=0$ if and only if ${\rm cap}_p(E\cap B;\lambda B)=0$ for all balls $B\subset X$ and some $\lambda>1$; and ${\rm cap}_p(B;\lambda B)$ is comparable to $r^{-p}\mu(B)$, where the comparison constant depends only on the data of $X$ and on $\lambda$.
\end{lemma}
\begin{definition} We say that the set $E\subset X$ is uniformly $p$-fat  if there exist constant $C_f>0$ and $r_0>0$ such that for all $x\in E$ and $0<r<r_0$, we have
$$
{\rm cap}_p(E\cap B(x,r);B(x,2r))\geq C_f\ {\rm cap}_p(B(x,r); B(x,2r)).
$$
\end{definition}

We point out that, as it happens for Poincar\'e inequalities, we have a self-improving property for $p$-fatness condition. For a complete proof of the next result, we refer the reader to \cite{BMS}.
\begin{proposition}[\cite{BMS}, Theorem 1.2] \label{q0fat} Let $X$ be a proper, LLC, doubling metric measure space supporting a $(1,s)$-Poincar\'e inequality for some $s$ with $1\leq s<\infty$. Let $p>s$ and suppose that $E$ is uniformly $p$-fat. Then there exists $p_0<p$ so that $E$ is  uniformly $p_0$-fat.
\end{proposition}

The following proposition is a capacity version of the Sobolev-Poincar\'e inequality in Remark \ref{Lemma 2.4EMThesis}, also referred as Maz'ya type estimate. The proof is a straightforward generalization of the Euclidean case and it can be found in \cite{B}.

\begin{proposition}[\cite{B}, Proposition 3.2]\label{prop2.3}Let $X$ be a doubling metric measure space supporting a weak $(1,s)$-Poincar\'e inequality. Then there exists $C$ and $\lambda\geq 1$ such that for all balls $B$ in $X$, $u\in N^{1,s}(X)$ and $S=\lbrace x\in\frac{1}{2}B: u(x)=0\rbrace$, then
\begin{equation}
\left(\dashint_B\vert u\vert^t \, \dd\mu\right)^{\frac{1}{t}}\leq\left(\frac{C}{{\rm cap}_s(S,B)}\int_{\lambda B}g^s \, \dd\mu\right)^{\frac{1}{s}},
\end{equation}
for $C$ depending on $C_{PI}$ and $t$ as in Corollary \ref{coropoinca}.
\end{proposition}

\subsection{Newtonian spaces with zero boundary values}\label{zero boundary values} Let $\Omega$ be an open and bounded subset of $X$. We define $N^{1,q}_0(\Omega)$ to be the set of functions $u\in N^{1,q}(X)$ that are zero on $X\setminus\Omega$ $q$-q.e. The space $N_0^{1,q}(\Omega)$ is equipped with the norm $\Vert\cdot\Vert_{N^{1,q}}$. Note also that if $C_q(X\setminus\Omega) = 0$, then $N^{1,q}_0(\Omega)=N^{1,q}(X)$. We shall therefore always assume that $C_q(X \setminus \Omega) > 0$.

The following remark is an important consequence of the $(1,s)$-Poincar\'e inequality, it gives a useful Sobolev inequality for functions vanishing outside a ball $B$, see \cite{KS}.
\begin{remark}\label{Lemma 2.4EMThesis} There exist $C$ and $t>1$ such that for all balls $B$ in $X$ with radius $r<\frac{{\rm diam} X}{3}$ and all $u\in N_{0}^{1,s}(B)$ we have
\begin{equation*}
		\left(\dashint_{B} |u|^t  \, \dd\mu\right)^{\frac{1}{t}}\leq C r \left(\dashint_{B}g_u^p \, \dd\mu\right)^{\frac{1}{p}},
	\end{equation*}
	where $C$ depends on $C_{PI}$ and $t$ is as in Remark \ref{rem1poin}. \end{remark}
\noindent In order to avoid clumsy notation, we assume that ${\rm diam} \,X=\infty$, i.e., that the above Sobolev inequality holds for all balls. In the opposite case, some of the results in this paper only hold for small balls whose radius depends on ${\rm diam}\, X$.

Next, we present some useful results concerning Newtonian spaces with zero boundary values. Proposition \ref{hardyineq} provides a characterization for $N_0^{1,q}$-functions by means of the Hardy inequality. Lemma \ref{compactnessr} gives a sufficient condition for a sequence of $N_0^{1,q}$-functions to converge to a $N_0^{1,q}$-function. Finally, Proposition \ref{iintersection} shows that $N_0^{1,q}$ can be presented as an intersection of $N_0^{1,q}$ and of zero Newtonian spaces with lower exponents. For further details and the proofs of the next results we refer the reader to \cite{MZG}.

\begin{proposition}[\cite{MZG}, Proposition 2.5] \label{hardyineq} Let $X$ be a proper, doubling, LLC metric measure space supporting a weak $(1,s)$-Poincar\'e inequality for some $1<s<q$, and suppose that $\Omega$ is a bounded domain in $X$ such that $X\setminus \Omega$ is uniformly $q$-fat. Then there is a constant $C$, depending only on $\Omega$ and $q$, such that $u\in N^{1,q}(X)$ is in $N^{1,q}_0(\Omega)$ if and only if
$$
\int_{\Omega}\left(\frac{\vert u(x)\vert}{{\rm dist}(x, X\setminus \Omega)}\right)^q\, \dd\mu\leq C \int_{\Omega}g_u(x)^q\, \dd\mu.
$$
\end{proposition}
\begin{remark}\label{C uniformly bound}
The constant $C$ in the above proposition formally depends on $q$. However, if $q$ varies inside a bounded interval, then the arguments in the proof of Proposition \ref{hardyineq} show that the appropriate constants are uniformly bounded. For this reason, since in our case all exponents vary inside a bounded interval $(s,s^*)$ we omit the dependence of the constant on $q$.
\end{remark}

\begin{lemma}[\cite{MZG}, Lemma 2.6]\label{compactnessr}
In the setting of Proposition \ref{hardyineq}, let $u_i\in N^{1,q}_0(\Omega)$ be a sequence that is bounded in $ N^{1,q}_0(\Omega)$. If $u_i\rightarrow u \ \mu$-a.e., then $u\in  N^{1,q}_0(\Omega)$.
\end{lemma}
\noindent We note that Lemma \ref{compactnessr} was originally formulated in \cite{KKM} for $(X,d,\mu)$ doubling and for $\Omega$ open such that $X\setminus\Omega$ satisfies a measure thickness assumption. Even though a measure thickness condition is stronger than a fatness assumption, the lemma still follows from Proposition \ref{hardyineq} (for further details see \cite{MZG} and the references therein). The assertion of the next proposition depends on the set $\Omega$. Even in $\mathbb{R}^n$ some type of thickness assumption on the domain is needed (see \cite{HedKil}).

\begin{proposition}[\cite{MZG}, Proposition 2.7]\label{iintersection} Let $X$ be a proper, doubling, LLC metric measure space supporting a weak $(1,s)$-Poincar\'e inequality for some $1<s<q$ and suppose that $\Omega$ is a bounded domain in $X$ such that $X\setminus\Omega$ is uniformly $q$-fat. Then
\begin{equation}\label{intersection}
N^{1,q}_0(\Omega)=N^{1,q}(\Omega)\cap \bigcap_{\epsilon>0}N^{1,q-\epsilon}_0(\Omega).
\end{equation}
\end{proposition}

\textit{Throughout this paper, we suppose that $(X, d, \mu)$ is a complete, locally linearly connected (LLC), metric measure space with metric $d$ and a doubling Borel regular measure $\mu$. We work on $\Omega\subset X$, an open and bounded subset such that $X \setminus \Omega$ is of positive $q$-capacity and uniformly $p$-fat, with $1<p<q$. Moreover, we assume that $X$ supports a weak $(1,p)$-Poincar\'{e} inequality. From now on and without further notice, we fix $1<s<p<q<s^*$ for which $X$ also admits a weak  $(1,s)$-Poincar\'e  inequality. Such $s$ is given by Theorem \ref{spoincare} and will be used in various of our results.}

\section{$(p,q)$-Quasiminimizers}\label{newsection3}
In this note, we are interested in the $(p,q)$-Dirichlet integral given by \eqref{J}, for some exponents $1<p<q$. As aforementioned, since we work in metric measure spaces and due to the methods we use, we treat it under sharp assumptions. That is, we assume that the coefficient functions $a,b:X \to \mathbb{R}$ are bounded and measurable with $0<\alpha\leq a,b \leq \beta$, for some positive constants $\alpha, \beta$. Now, we introduce the definition of $(p,q)$-quasiminimizers of integral \eqref{J}.
\begin{definition}\label{qm}
	A function $u\in N^{1,q}({\Omega})$ is a $(p,q)$-quasiminimizer on $\Omega$ if there exists $K>0$, called quasiminimizing constant, such that for every open $\Omega'\Subset\Omega$ and every test function $v \in N^{1,q}({\Omega'})$ with $u-v\in  N^{1,q}_0({\Omega'})$ the inequality
	\begin{align}\label{min}
		\int_{\Omega'} (a g_{u}^p+bg_{u}^q)\, \dd \mu
		\leq& K \int_{\Omega'}(a g_{v}^p\, + b g_{v}^q) \,\dd \mu 
	\end{align} 
	holds,  where $g_{u}$, $g_{v}$ are the minimal $q$-weak upper gradients of $u$ and $v$ in ${\Omega}$, respectively.
	Furthermore, a function $u\in N^{1,q}({\Omega})$ is a global $(p,q)$-quasiminimizer on $\Omega$ if \eqref{min} is satisfied with $\Omega$ instead of $\Omega'$, for all $v \in N^{1,q}({\Omega})$, with $u-v\in  N^{1,q}_0({\Omega})$.
\end{definition}

From now on, to simplify notation we refer to global $(p,q)$-quasiminimizers by just writing $(p,q)$-quasiminimizers.

Next, we report the De Giorgi Lemma, which has a key role in our paper and whose proof can be found in \cite{NP}.
In order to do so, we start introducing some notations. We denote $S_{k, r}= \{x \in B(y,r) \cap {\Omega}: u(x)>k\},$ where $k \in \mathbb{R}$ and $r>0$.  Also, for every $y \in \Omega$, we define $R(y)=\frac{d(y,X\setminus \Omega)}{2}$.
\begin{lemma}[\cite{NP}, Lemma 3.1]\label{lem 4.1}
	Let $u\in N^{1,q}({\Omega})$ be a $(p,q)$-quasiminimizer. If $0<\rho<R< R(y)$, then there exists $C$ such that the following De Giorgi type inequality 
	\begin{equation}\label{5.4}
		\int_{S_{\alpha, \rho}}(ag_{u} ^p+bg_{u} ^q)\, \dd \mu	\leq C\left( \frac{1}{(R-\rho)^{p}} \int_{S_{k, R}} a(u-k) ^p\, \dd \mu+\frac{1}{(R-\rho)^{q}} \int_{S_{k, R}}b(u-k)^q \,\dd \mu \right),
	\end{equation}
	is satisfied. The constant $C$ depends on $K$, given by Definition \ref{qm}, and $q$.
\end{lemma}

\begin{remark}
	We note that inequality \eqref{5.4} is equivalent to the following
	\begin{align}\label{Cacioppoli}
		\int_{B(y,\rho)}(ag_{u} ^p+bg_{u} ^q)\, \dd \mu	& \leq C \Bigg( \frac{1}{(R-\rho)^{p}} \int_{B(y,R)} a(u-k) ^p_+ \,\dd \mu \nonumber \\  & \hspace{1.5cm}+\frac{1}{(R-\rho)^{q}} \int_{B(y,R)}b(u-k)^q_+ \,\dd \mu\Bigg),
	\end{align} where $(u-k)_+= \max\{u-k,0\}$.
\end{remark}

\section{Higher integrability property }\label{Sec3}

This section is devoted to prove global higher integrability for upper gradients of $(p,q)$- quasiminimizers with fixed boundary data belonging to a slightly better Newtonian space. Let $u\in N^{1,q}(\Omega)$, we say that $u$ is a $(p, q)$-quasiminimizer with boundary data $w\in N^{1,q}(\Omega)$, if $w-u\in N^{1,q}_{0}(\Omega)$. 
A complete survey on global higher integrability of gradients in the Euclidean case can be found in the book by Kinnunen, Lehrb\"{a}ck and V\"{a}h\"{a}kangas in \cite{KLM} for weak solutions of $p$-Laplace equation. In the general metric setting, the improvement of integrability is obtained by using a metric version of the Gehring Lemma, whose proof can be found, for example, in \cite{Maa} or \cite{ZG}. We remark that this lemma holds in all doubling metric measure spaces. 

\begin{lemma}\label{Gehring} Let $\sigma\in [s_0,s_1]$, where $s_0,s_1>1$ are fixed. Let $g\in L^{\sigma}_{\textrm{loc}}(X)$ and $f\in L^{s_1}_{\textrm{loc}}(X)$ be non-negative functions. Assume that there exists constant $C_G>1$ such that for every ball $B\subset \lambda B\subset X$ the following inequality
	\begin{equation*}
		\dashint_{B}g^{\sigma} \,\dd\mu\leq C_G\left[\left(\dashint_{\lambda B} g \,\dd\mu\right)^\sigma+\dashint_{\lambda B} f^\sigma \,\dd\mu\right]
	\end{equation*}
	holds for some $\sigma >1$. Then there exists $\epsilon_0>0$ such that $g\in L^{\tilde{s}}_{\textrm{loc}}(X,\mu)$ for $\tilde{s}\in[\sigma, \sigma+\epsilon_0[$ and moreover
	\begin{equation*}
		\left(\dashint_{B}g^{\tilde{s}} \,\dd\mu\right)^{\frac{1}{\tilde{s}}}\leq C\left[\left(\dashint_{\lambda B} g \,\dd\mu\right)^\sigma+\left(\dashint_{\lambda B} f^{\tilde{s}} \,\dd\mu\right)^{\frac{1}{\tilde{s}}}\right],
	\end{equation*}
	for $\epsilon_0$ and $C$ depending on $s_0,s_1, \lambda, C_d$ and $C_G$.
\end{lemma}
Now, we state the global higher integrability result for the minimal weak upper gradient of a $(p,q)$-quasiminimizer. As in \cite{MZG}, the proof is based on the capacity version of the Sobolev-Poincar\'e inequality in Proposition \ref{prop2.3}, which implies a reverse H\"older inequality for the minimal weak upper gradient. Global higher integrability then follows from Lemma \ref{Gehring}. 
\begin{theorem}\label{Theorem 5.2}
	Let $w \in N^{1,\bar{q}}(\Omega)$ for some $\bar{q}>q$. If $u \in N^{1,q}(\Omega)$ is a $(p,q)$-quasiminimizer with boundary data $w$, then there exists $\delta_0\in]0,q-\bar{q}[$ such that $g_u \in L^{q+\delta}(\Omega)$ for all $\delta \in ]0, \delta_ 0[$ and
	\begin{equation*}
		\left(\dashint_{\Omega}g_u^{q+\delta} \, \dd\mu\right)^{\frac{1}{q+\delta}}\leq C \left(\left(	\dashint_{\Omega}g_u^{q} \,\dd\mu\right)^{\frac{1}{q}}+ \left(		\dashint_{\Omega} g_w^{q+\delta} \,\dd\mu\right)^{\frac{1}{q+\delta}} +1\right),
	\end{equation*}
	where $\delta_0$ and $C$ depend on $p$ and $q$. 
\end{theorem}

\begin{proof}
	We note that, by Remark \ref{rem1poin}, a $(1,s)$-Poincar\'{e} inequality with $1<s<p<q<s^*$ implies both a $(p,s)$ and a $(q,s)$-Poincar\'{e} inequalities. Moreover, we remark that  $X\setminus \Omega$ is uniformly $p$-fat. By Proposition \ref{q0fat}, then $X\setminus \Omega$ is also uniformly $p_0$-fat for some $p_0<p$.
	Without loss of generality, we can assume $p_0\geq s$. Indeed, if $p_0<s$ we can choose a bigger $p_0$ so that $p_0=s$. Thus, the H\"older inequality implies $(p, p_0)$ and $(q,p_0)$-Poincar\'e inequalities. Let $B_0$ be a ball in $X$ such that $\overline{\Omega} \subset  B_0 \subset 2B_0$. Now we consider a ball $B = B(x_0, r)$ with fixed radius $r > 0$ satisfying $2\lambda B \subset 2B_0$, where $\lambda$ is the dilation factor in Definition \ref{PI}.
	
	First,	let $2\lambda B \subset \Omega$. By \eqref{Cacioppoli} with $\alpha= u_{2B}$, \eqref{doubling}, $(p, p_0)$ and $(q, p_0)$-Poincar\'{e} inequalities we obtain
	\begin{align*}
		\alpha\dashint_B (g_u^p +  g_u^q)\, \dd\mu &\leq\dashint_B (a g_u^p + b g_u^q)\, \dd\mu \\&\leq C \left( \frac{1}{r^{p}} \dashint_{2B} a|u-u_{2B}|^p \,\dd \mu+\frac{1}{r^{q}} \dashint_{2B}b|u-u_{2B}|^q \,\dd \mu \right)\nonumber \\
		&\leq C \left( \frac{1}{r^{p}} \dashint_{2B} \beta |u-u_{2B}|^p \,\dd \mu+\frac{1}{r^{q}} \dashint_{2B}\beta |u-u_{2B}|^q \,\dd \mu \right)\nonumber \\
		&= C \left( \frac{1}{r^{p}} \dashint_{2B}  |u-u_{2B}|^p \,\dd \mu+\frac{1}{r^{q}} \dashint_{2B} |u-u_{2B}|^q \,\dd \mu \right)\nonumber \\
		&\leq C \left(\left(\dashint_{2 \lambda B} g_u^{p_0}  \,\dd \mu\right)^{\frac{p}{p_0}}+\left(\dashint_{2\lambda B}g_u^{p_0} \,\dd \mu \right)^{\frac{q}{p_0}}\right),
	\end{align*}
	Therefore,
	\begin{align*}
		\dashint_B (g_u^p +  g_u^q)\, \dd\mu 
		&\leq C \left(\left(\dashint_{2 \lambda B} g_u^{p_0} \,\dd \mu\right)^{\frac{p}{p_0}}+\left(\dashint_{2\lambda B}g_u^{p_0} \,\dd \mu \right)^{\frac{q}{p_0}}\right),
	\end{align*}
	where $C$ depends on $q,\alpha, \beta, K, C_d$ and $C_{PI}$.
	
	\noindent We recall Young's inequality $|z|^q \leq |z|^p+|z|^q\leq C(1+|z|^q),$ with $C=\frac{p+q}{q}< 2$.
	Thus, 
	\begin{align}\label{3.2}
		\dashint_B g_u^q\, \dd\mu 
		&\leq C \left(1+\left(\dashint_{2 \lambda B} g_u^{p_0} \,\dd \mu\right)^{\frac{q}{p_0}}\right),
	\end{align}
	where $C$ depends on $q,\alpha, \beta, K, C_d$ and $C_{PI}$.
	
	Secondly, let $2\lambda B \setminus \Omega\neq \emptyset$. We consider a Lipschitz cut-off function such that
	$0\leq\eta \leq 1$, $\eta=1$ on $B$, ${\rm supp} \, \eta\subset 2B$ and $g_{\eta}\leq \frac{C}{r}$. Therefore, $\eta (u-w) \in  N^{1,q}_0 (2B \cap\Omega)$ and so we can plug it in \eqref{min}. We get
	\begin{equation*}
		\int_{2B \cap \Omega}(a g_u^p + b g_u^q) \, \dd\mu \leq K\int_{2B \cap \Omega} (a g_v^p + b g_v^q) \, \dd\mu,
	\end{equation*}
	where $v =u+\eta(w-u)$ and $g_v\leq (1-\eta)\,g_u+\eta \,g_w+|u-w|g_{\eta}$. Thus, we have
	\begin{align*}
		\int_{B \cap \Omega} (a g_u^p +b g_u^q) \, \dd\mu &\leq C \int_{(2B\setminus B)\cap \Omega}a(1-\eta)^p\,g_u^p  \, \dd\mu + C \int_{2B \cap \Omega} a |u-w|^pg_{\eta}^p \, \dd\mu \\&\quad+ C \int_{2B \cap \Omega}a\eta^p \,g_w^p \, \dd\mu + C \int_{(2B\setminus B)\cap \Omega}b(1-\eta)^q\,g_u^q  \, \dd\mu \\&\quad+ C \int_{2B \cap \Omega}b |u-w|^qg_{\eta}^q \, \dd\mu + C \int_{2B \cap \Omega}b\eta^q \,g_w^q \, \dd\mu.
	\end{align*}
	Now, we add $C	\int_{B \cap \Omega} (ag_u^p +bg_u^q) \, \dd\mu $ to both sides of the previous inequality and then we divide by $(1 + C)$, obtaining
	\begin{align*}
		\int_{B \cap \Omega} (ag_u^p +bg_u^q) \, \dd\mu &\leq \theta \int_{2B\cap \Omega}(a g_u^p + b g_u^q) \, \dd\mu+ \frac{\theta}{r^p} \int_{2B \cap \Omega} a|u-w|^p \,\dd\mu \\&\quad+  \frac{\theta}{r^q} \int_{2B \cap \Omega}b |u-w|^q \, \dd\mu + \theta \int_{2B\cap \Omega}(ag_w^p +bg_w^q) \, \dd\mu,
	\end{align*}
	where $\theta=\frac{C}{1+C}<1$.
	
	By Lemma 6.1 of \cite{G}, we get
	\begin{align*}
		\int_{B \cap \Omega} (a g_u^p +b g_u^q) \, \dd\mu &\leq  \frac{C}{r^p} \int_{2B \cap \Omega} a|u-w|^p \, \dd\mu +  \frac{C}{r^q} \int_{2B \cap \Omega}b |u-w|^q \, \dd\mu \nonumber \\&\quad+ C \int_{2B\cap \Omega}(a g_w^p +b g_w^q) \, \dd\mu.
	\end{align*}
	Notice that the last inequality can be rewritten as 
	\begin{align}\label{3.3}
		\int_{B \cap \Omega} (g_u^p +g_u^q) \, \dd\mu &\leq  \frac{C}{r^p} \int_{2B \cap \Omega} |u-w|^p \, \dd\mu +  \frac{C}{r^q} \int_{2B \cap \Omega} |u-w|^q \, \dd\mu \nonumber \\&\quad+ C \int_{2B\cap \Omega}( g_w^p + g_w^q) \, \dd\mu,
	\end{align} where $C$ depends on $q,\alpha, \beta$ and $K$. From now on, for the integrals on the right-hand side of the previous inequality we will work on the bigger ball $4B$. By Proposition \ref{prop2.3} with $s=p_0$ and the doubling condition, for $l=p,q$, we deduce
	\begin{align*}
		\left(\frac{C}{r^l}\dashint_{4B}|u-w|^l\, \dd\mu \right)^{\frac{1}{l}}&\leq \frac{C}{r}\left(\frac{1}{\textrm{cap}_{p_0}(S;4B)}\int_{4\lambda B}g_{u-w}^{p_0}\, \dd\mu \right)^{\frac{1}{p_0}}\\&\leq C\left(\frac{\mu(2B)r^{-p_0}}{\textrm{cap}_{p_0}(S;4B)}\dashint_{4\lambda B}g_{u-w}^{p_0}\, \dd\mu \right)^{\frac{1}{p_0}},
	\end{align*} 
	where $S=\lbrace x\in 2B: u(x)=w(x)\rbrace$. We note that $u-w=0$ $q$-q.e., therefore $u-w=0$ $p_0$-q.e. in $X\setminus\Omega$. Also, since $X\setminus\Omega$ is uniformly $p_0$-fat, we obtain the following chain of inequalities
	$$
	\textrm{cap}_{p_0}(S;4B)\geq\textrm{cap}_{p_0}(2B\setminus\Omega; 4B)\geq C_f\ \textrm{cap}_{p_0}(2B;4B)\geq C\mu(2B)r^{-p_0}.$$ Since $u-w=0$ $q$-q.e, then $\mu$-a.e. in $X \setminus\Omega$ and so $g_{u-w}=0$ $\mu$-a.e. in $X \setminus\Omega$. Thus, for $l=p,q$,
	\begin{align}\label{3.4}
		\left(\frac{C}{r^l}\frac{1}{\mu(4B)}\int_{4B\cap \Omega}|u-w|^l\, \dd\mu \right)^{\frac{1}{l}}&=\left(\frac{C}{r^l}\int_{4B\cap \Omega}|u-w|^l\, \dd\mu \right)^{\frac{1}{l}}\nonumber\\&\leq C\left(\dashint_{4\lambda B}g_{u-w}^{p_0}\, \dd\mu \right)^{\frac{1}{p_0}}\nonumber\\&= C\left(\frac{1}{\mu(4\lambda B)}\int_{4\lambda B\cap\Omega}g_{u-w}^{p_0}\, \dd\mu \right)^{\frac{1}{p_0}}\nonumber \\
		&\leq C\left(\frac{1}{\mu(4B)}\int_{4\lambda B}g_{u}^{p_0}\, \dd\mu \right)^{\frac{1}{p_0}}+C\left(\frac{1}{\mu(4B)}\int_{4\lambda B}g_{w}^{p_0}\, \dd\mu \right)^{\frac{1}{p_0}}.
	\end{align}
	Now, H\"{o}lder inequality implies
	\begin{align}\label{3.5}
		\left(\frac{1}{\mu(4\lambda B)}\int_{4\lambda B \cap \Omega}g_{w}^{p_0}\, \dd\mu \right)^{\frac{1}{p_0}}&\leq 		\left(\frac{1}{\mu(4\lambda B)}\int_{4\lambda B \cap \Omega}g_{w}^l\, \dd\mu \right)^{\frac{1}{l}},
	\end{align} for $l=p,q$.
	Using \eqref{3.3}, \eqref{3.4} and \eqref{3.5}, we have
	\begin{align*}
		\frac{1}{\mu(B)}	\int_{B \cap \Omega} (g_u^p +g_u^q) \, \dd\mu \leq &C \Bigg(\frac{1}{\mu(4\lambda B)}\int_{4\lambda B\cap \Omega}(g_w^p + g_w^q)\,\dd\mu+ \left(\frac{1}{\mu(4\lambda B)}\int_{4\lambda B \cap \Omega}g_{u}^{p_0}\, \dd\mu \right)^{\frac{p}{p_0}}  \nonumber \\ &+\left(\frac{1}{\mu(4\lambda B)}\int_{4\lambda B \cap \Omega}g_{u}^{p_0}\, \dd\mu \right)^{\frac{q}{p_0}} +\frac{1}{\mu(4 B)}\int_{4B\cap \Omega}(g_w^p + g_w^q)\, \dd\mu\Bigg).
	\end{align*}
	By the doubling property, we get
	\begin{align}\label{3.6old}
	\frac{1}{\mu(B)}	\int_{B \cap \Omega} (g_u^p +g_u^q) \, \dd\mu \leq &C \Bigg(\frac{1}{\mu(4\lambda B)}\int_{4\lambda B\cap \Omega}(g_w^p + g_w^q)\,\dd\mu+ \left(\frac{1}{\mu(4\lambda B)}\int_{4\lambda B \cap \Omega}g_{u}^{p_0}\, \dd\mu \right)^{\frac{p}{p_0}}  \nonumber \\ &+\left(\frac{1}{\mu(4\lambda B)}\int_{4\lambda B \cap \Omega}g_{u}^{p_0}\, \dd\mu \right)^{\frac{q}{p_0}}\Bigg).
	\end{align}
	where $C$ depends on $q,\alpha, \beta, K, C_d$ and $C_f$.

\noindent By Young's inequality, we deduce
\begin{align*}
\frac{1}{\mu(4\lambda B)}\int_{4\lambda B\cap \Omega}(g_w^p + g_w^q)\, \dd\mu 
&\leq \frac{2}{\mu(4\lambda B)}\int_{4\lambda B\cap \Omega}(1+ g_w^q)\, \dd\mu\\
&= 2 \left(\dfrac{\mu(4\lambda B\cap \Omega)}{\mu(4\lambda B)} + \dfrac{1}{\mu(4\lambda B)}\int_{4\lambda B\cap \Omega}g_w^q\, \dd\mu \right)\\
&\leq 2 \left(1+ \dfrac{1}{\mu(4\lambda B)}\int_{4\lambda B\cap \Omega}g_w^q\, \dd\mu \right).
\end{align*}
Analogously, we get
\begin{align*}
 \left(\frac{1}{\mu(4\lambda B)}\int_{4\lambda B \cap \Omega}g_{u}^{p_0}\, \dd\mu \right)^{\frac{p}{p_0}} +\left(\frac{1}{\mu(4\lambda B)}\int_{4\lambda B \cap \Omega}g_{u}^{p_0}\, \dd\mu \right)^{\frac{q}{p_0}}\leq 2 \left(1+ \left(\frac{1}{\mu(4\lambda B)}\int_{4\lambda B \cap \Omega}g_{u}^{p_0}\, \dd\mu \right)^{\frac{q}{p_0}} \right).
\end{align*}
By \eqref{3.6old}, we obtain
\begin{align*}
	\frac{1}{\mu(B)}\int_{B \cap \Omega} (g_u^p +g_u^q) \, \dd\mu&\leq C \Bigg(1+ \dfrac{1}{\mu(4\lambda B)}\int_{4\lambda B\cap \Omega}g_w^q\, \dd\mu+\left(\frac{1}{\mu(4\lambda B)}\int_{4\lambda B \cap \Omega}g_{u}^{p_0}\, \dd\mu \right)^{\frac{q}{p_0}}\Bigg)\\
	&= C  \Bigg(\dashint_{4\lambda B}(g_w^q \,\chi_{4\lambda B\cap \Omega} +1) \,\dd\mu+ \left(\frac{1}{\mu(4\lambda B)}\int_{4\lambda B \cap \Omega}g_{u}^{p_0}\, \dd\mu \right)^{\frac{q}{p_0}}\Bigg).
\end{align*}
Thus,
\begin{align*}
	\frac{1}{\mu(B)}	\int_{B \cap \Omega} g_u^q \, \dd\mu
& \leq C  \left(\dashint_{4\lambda B}(g_w^q \,\chi_{4\lambda B\cap \Omega} +1) \,\dd\mu+ \left(\frac{1}{\mu(4\lambda B)}\int_{4\lambda B \cap \Omega}g_{u}^{p_0}\, \dd\mu \right)^{\frac{q}{p_0}}\right)\\
&\leq C  \left(\dashint_{4\lambda B}\left(g_w^{p_0} \,\chi_{4\lambda B\cap \Omega} +1\right)^{\frac{q}{p_0}} \,\dd\mu+ \left(\frac{1}{\mu(4\lambda B)}\int_{4\lambda B \cap \Omega}g_{u}^{p_0}\, \dd\mu \right)^{\frac{q}{p_0}}\right).
\end{align*}
And so,
\begin{align}\label{3.6}
	\frac{1}{\mu(B)}	\int_{B \cap \Omega} g_u^q \, \dd\mu\leq C  \Bigg(&\dashint_{4\lambda B}\left(g_w^{p_0} \,\chi_{4\lambda B\cap \Omega} +1\right)^{\frac{q}{p_0}}\, \dd\mu+\left(\frac{1}{\mu(4\lambda B)}\int_{4\lambda B \cap \Omega}g_{u}^{p_0}\, \dd\mu \right)^{\frac{q}{p_0}}\Bigg).
\end{align}
We define $\sigma=\frac{q}{p_0}>1$,
	\begin{equation*}
		g=\begin{cases}
			g_u^{p_0} \quad\mbox{in $\Omega$},\\
			0 \quad\mbox{otherwise},
		\end{cases}
	\end{equation*}
	and
	\begin{equation*}
		f=\begin{cases}
			g_w^{p_0} \,\chi_{4\lambda B\cap \Omega} +1 \quad\mbox{in $\Omega$},\\
			0 \hspace{2.54cm}\mbox{otherwise}.
		\end{cases}
	\end{equation*}
If $4 \lambda B \subset 2B_0$, then by \eqref{3.2} and \eqref{3.6}, the following reverse H\"{o}lder type inequality holds
	\begin{align*}
		\dashint_{B} g^{\sigma}\,\dd\mu\leq&C \left(\left(\dashint_{4\lambda B}g\,\dd\mu\right)^{\sigma}+ 	\dashint_{4\lambda B}f^{\sigma} \,\dd\mu\right),
	\end{align*}
where $C$ depends on $q,\alpha, \beta, K, C_d, C_f$ and $C_{PI}$.	
By applying Gehring Lemma \ref{Gehring}, we obtain the inequality
	\begin{equation*}
		\left(	\dashint_{B} g^{\tilde{s}}\,\dd\mu\right)^{\frac{1}{\tilde{s}}} \leq C \left(\left(	\dashint_{4\lambda B}g^{\sigma} \,\dd\mu\right)^{\frac{1}{\sigma}}+ \left(	\dashint_{4\lambda B} f^{\tilde{s}} \,\dd\mu\right)^{\frac{1}{\tilde{s}}} \right),
	\end{equation*}
	for $\tilde{s}\in [\sigma,\sigma+\epsilon_0[$, with $C$ and $\epsilon_0$ both depending on $p,q, \alpha, \beta, K, \lambda, C_d,C_f$ and $C_{PI}$.	
	That is,
	\begin{align*}
		\left(	\dashint_{B} g_u^{p_0\tilde{s}}\,\dd\mu\right)^{\frac{1}{\tilde{s}}}& 
		\leq C \left(\left(\dashint_{4\lambda B}g_u^{q} \,\dd\mu\right)^{\frac{p_0}{q}}
		+ \left(\dashint_{4\lambda B} 	\left(g_w^{p_0} \,\chi_{4\lambda B\cap \Omega} +1\right)^{\tilde{s}} \,\dd\mu\right)^{\frac{1}{\tilde{s}}} \right)\\
		& \leq C \left(\left(	\dashint_{4\lambda B}g_u^{q} \,\dd\mu\right)^{\frac{p_0}{q}}+ \left(	\dashint_{4\lambda B} 2^{\tilde{s}-1}	\left(g_w^{p_0\tilde{s}} \,\chi_{4\lambda B\cap \Omega} +1\right)\,\dd\mu\right)^{\frac{1}{\tilde{s}}} \right)\\
		& \leq C \left(\left(	\dashint_{4\lambda B}g_u^{q} \,\dd\mu\right)^{\frac{p_0}{q}}+ \left(	\dashint_{4\lambda B}\left(g_w^{p_0\tilde{s}} \,\chi_{4\lambda B\cap \Omega} +1\right)\,\dd\mu\right)^{\frac{1}{\tilde{s}}} \right)\\
			& = C \left(\left(	\dashint_{4\lambda B}g_u^{q} \,\dd\mu\right)^{\frac{p_0}{q}}+ \left(	\dashint_{4\lambda B}g_w^{p_0\tilde{s}} \,\chi_{4\lambda B\cap \Omega} \,\dd\mu +1\right)^{\frac{1}{\tilde{s}}} \right)\\
				& \leq C \left(\left(	\dashint_{4\lambda B}g_u^{q} \,\dd\mu\right)^{\frac{p_0}{q}}+ \left(	\dashint_{4\lambda B}g_w^{p_0\tilde{s}} \,\chi_{4\lambda B\cap \Omega}\, \dd\mu \right)^{\frac{1}{\tilde{s}}}+1 \right)\\
				& \leq C \left(\left(	\dashint_{4\lambda B}g_u^{q} \,\dd\mu\right)^{\frac{p_0}{q}}+ \left(	\dfrac{1}{\mu(4\lambda B)}\int_{4\lambda B \cap \Omega}g_w^{p_0\tilde{s}} \,\dd\mu \right)^{\frac{1}{\tilde{s}}}+1 \right).
	\end{align*}
	where $q \leq q+\delta_0=p_0\tilde{s} \leq q+\epsilon_0p_0$, with $\delta_0=\epsilon_0p_0$, where $\delta_0$ has the same dependencies as $\epsilon_0$, and $\delta\in [0, \delta_0]$.
	Therefore, we get 
	\begin{align}\label{(3.7)1}
		\left(	\dashint_{B} g_u^{q+\delta} \,\dd\mu\right)^{\frac{1}{q+\delta}} &\leq C \left(\left(	\dashint_{4\lambda B}g_u^{q} \,\dd\mu\right)^{\frac{1}{q}}+ \left(	\dfrac{1}{\mu(4\lambda B)}\int_{4\lambda B \cap \Omega} g_w^{q+\delta} \,\dd\mu\right)^{\frac{1}{q+\delta}} +1\right)\nonumber\\
		& \leq C \left(\left(	\dashint_{4\lambda B}g_u^{q} \,\dd\mu\right)^{\frac{1}{q}}+ \left(		\dashint_{4\lambda B} g_w^{q+\delta} \,\dd\mu\right)^{\frac{1}{q+\delta}} +1\right),
	\end{align} for all $\delta \in [0,\delta_0]$.

\noindent We recall that $\Omega$ is bounded and its diameter is finite, so we can cover it with a finite number of balls $B(x_j, r_j)$, $j = 1, 2, ...,N$, that is
\begin{equation*}
    B(x_j, 2\lambda r_j)\subset B_0\quad\textrm{and}\quad\Omega \subset \bigcup_{j=1}^NB(x_j, r_j),  
\end{equation*}
	for a fixed $\lambda$. Now, we multiply \eqref{(3.7)1} by $\mu(4\lambda B)^{\frac{1}{q+\delta}}$, we sum over $B(x_j, r_j)$ and divide by $\mu(\Omega)$, in order to complete the proof. We note that this last step may slightly modify the constant $C$, however this will only depend on $C_d$ and $\Omega$ itself. Finally, we point out that if $\alpha$, $\beta$, $K$, $\lambda$, $C_d$, $C_f$ and $C_{PI}$ are fixed, then $C$ and $\delta_0$ depend essentially only on $p$ and $q$.
\end{proof}

\section{Stability property}\label{Sec4}
In this section, we prove a stability result, with respect to the varying exponents $p_i,q_i$, for a family of quasiminimizers. We show that, under suitable assumptions, a sequence of $(p_i,q_i)$-quasiminimizers converge to a $(p,q)$-quasiminimizer of the limit problem. The global higher integrability in Theorem \ref{Theorem 5.2} serves as a starting point for the proof of the stability. Before stating our main theorem, we premise the following remark.
\begin{remark}\label{things}Under the assumptions of the Theorem \ref{Theorem 5.2}, let $(p_i,q_i)$ be a sequence such that  $\lim_{i\rightarrow\infty}p_i=p$ and $\lim_{i\rightarrow\infty}q_i=q$. Without loss of generality we can assume that $1<s<p_i<q_i<s^*$ for every $i\in \mathbb{N}$.
By Theorem  \ref{Theorem 5.2}, for every duple $(p_i,q_i)$ there exists $\delta_i$ such that if $u_i\in N^{1,q_i}(\Omega)$ is a $(p_i, q_i)$-quasiminimizer, then the minimal $q_i$-weak upper gradient $g_{u_i}$ belongs to the space $L^{q_i+\delta_i}(\Omega)$ and the following inequality is satisfied

\begin{equation}\label{primera}
	\left(\dashint_{\Omega} g_{u_i}^{q_i+\delta_i}\,\dd\mu\right)^{\frac{1}{q_i+\delta_i}} \leq C_i \left(\left(	\dashint_{\Omega}g_{u_i}^{q_i} \,\dd\mu\right)^{\frac{1}{q_i}}+ \left(	\dashint_{\Omega} g_w^{q_i+\delta_i} \,\dd\mu\right)^{\frac{1}{q_i+\delta_i}}+1 \right),
\end{equation} 
where $\delta_i$ and $C_i$ depend on $p_i$ and $q_i$. Since $1<s<p_i<q_i<s^*$, we can find $C$, depending on $p$ and $q$, such that $C_i\leq C$. Moreover, in \eqref{primera}, $\delta_i$ is inverse proportional to $C_i$, therefore there exists $\delta_0$, with the same dependencies as $C$, such that $\delta_i\geq\delta_0$. For more details, we refer the reader to \cite{ZG}.
\end{remark}

The main theorem of this section is the following stability result. Unlike the Euclidean case, we need to assume a priori that the sequence of $(p_i,q_i)$-quasiminimizers converges to a certain $u$, because when working with quasiminimizers, instead of minimizers, we loose the uniqueness of the limit. 

\begin{theorem}\label{Theorem 1.1.} Let $w \in N^{1,\bar{q}}(\Omega)$ for some $\bar{q} > q$. 
Assume $p = \lim_{i\to \infty} p_i$, $q = \lim_{i\to \infty} q_i$, with $p_i,q_i\in ]s,s^*[$ and $p_i<q_i$. Let $(u_i)$ a sequence of $(p_i,q_i)$-quasiminimizers, where $u_i\in N^{1,q_i}(\Omega)$ with the same boundary data $w$ and equal quasiminimizing constant $K$. If 
$u_i\to u$ $\mu$-a.e. in $\Omega$, then $u\in N^{1,q}(\Omega)$ is a $(p,q)$-quasiminimizer with boundary data $w$.
\end{theorem}

We note that functions $u_i$ are supposed to be not equal to the boundary data $w$, i.e. we assume that there is a set of positive measure where $u_i \neq w$ $\mu$-a.e., otherwise the result is trivial. Since the proof of Theorem \ref{Theorem 1.1.} is considerably long,  we have divided it  into several lemmas, which are of independent interest.

\begin{lemma}\label{Lemma4.2}
Let $u_i$ and $u$ be as in Theorem \ref{Theorem 1.1.}. Then there exists $\epsilon_0>0$ such that $u,u_i\in L^{q+\epsilon_0}(\Omega)$, $g_{u_i}, g\in  L^{q+\epsilon_0}(\Omega)$ and there is a subsequence such that $u_i\rightarrow u$ in $L^{q+\epsilon_0}(\Omega)$, $g_{u_i}\rightarrow g$ in $L^{q+\epsilon_0}(\Omega)$, where $g$ is a $q$-weak upper gradient of $u$. Here, $\epsilon_0$ depends on $p$ and $q$.
\end{lemma}
\noindent We notice that in the previous lemma we have convergence to some $q$-weak upper gradient of $u$ and not necessarily to the minimal $q$-weak upper gradient $g_u$. It is not even known whether the sequence of minimal upper gradients converges weakly to the $q$-minimal upper gradient of $u$, or not. Nevertheless, the lemma implies that 
\begin{equation*}
\Vert g_u\Vert_{L^q(\Omega)}\leq\liminf_{i\rightarrow\infty}\Vert g_{u_i}\Vert_{L^q(\Omega)}.
\end{equation*}
The proof of Lemma \ref{Lemma4.2} makes use of the metric version of Rellich-Kondrachov Theorem, see \cite{MZG}, which we report below for the reader's convenience.

\begin{theorem}[\cite{MZG}, Theorem 4.1]\label{rellich} Let $(X,d,\mu)$ be a metric space, where $\mu$ is doubling. Suppose that all the pairs $(u_i, g_i)$, satisfy a weak $(1,p)$-Poincar\'e inequality. Fix a ball $B$ and assume that the sequence $\Vert u_i\Vert_{L^1(B)}+\Vert g_i\Vert_{L^1(5\tau B)}$ is bounded. Then there is a subsequence of $(u_i)$ that converges in $L^q(B)$ for each $1\leq q\leq \frac{Qp}{Q-p}$, when $p<Q$, and for each $q\geq 1$, when $p\geq Q$. Here $Q=\log C_d$ and $C_d$ is the doubling constant of $\mu$.
\end{theorem}
\begin{proof}[Proof of Lemma \ref{Lemma4.2}] By the quasiminimizing property \eqref{min} of $u_i$ and using the boundary data $w$ as a test function, we obtain
\begin{align*}
\int_{\Omega}(ag_{u_i}^{p_i}+bg_{u_i}^{q_i})\,\dd\mu&\leq K\int_{\Omega}(ag_{w}^{p_i}+bg_{w}^{q_i})\,\dd\mu\\
&\leq K\int_{\Omega}\beta (g_{w}^{p_i}+ g_{w}^{q_i})\,\dd\mu\\
&\leq C\int_{\Omega}(g_{w}^{p_i}+g_{w}^{q_i})\,\dd\mu,
\end{align*}
where $C$ depends on $\beta$ and  $K$.
This implies, 
\begin{equation*}
\int_{\Omega}(g_{u_i}^{p_i}+g_{u_i}^{q_i})\,\dd\mu\leq C\int_{\Omega}(g_{w}^{p_i}+ g_{w}^{q_i})\,\dd\mu,
\end{equation*}
with $C$ depending also on $\alpha$.
By Young's inequality, we deduce
\begin{equation*}
g_w^{p_i}+g_w^{q_i}\leq\left(\frac{p_i+q_i}{q_i}\right)(1+ g_w^{q_i})\leq 2(1+ g_w^{q_i}).
\end{equation*}
Since, trivially $g_{u_i}^{q_i}\leq g_{u_i}^{p_i}+g_{u_i}^{q_i}$, we get
\begin{equation*}
\dashint_{\Omega}g_{u_i}^{q_i}\, \dd\mu\leq C\dashint_{\Omega}(1+g_w^{q_i})\,\dd\mu
=C\left(\dashint_{\Omega}g_w^{q_i}\,\dd\mu+1\right).
\end{equation*}
Therefore,
\begin{align*}
\left(\dashint_{\Omega}g_{u_i}^{q_i}\,\dd\mu\right)^{\frac{1}{q_i}}&\leq C^{\frac{1}{q_i}}\left(\dashint_{\Omega}g_w^{q_i}\,\dd\mu+1\right)^{\frac{1}{q_i}}\\
&\leq C_i\left(\left(\dashint_{\Omega}g_w^{q_i}\,\dd\mu\right)^{\frac{1}{q_i}}+1\right),
\end{align*}
where now $C_i$ depends also on $q_i$. By Theorem \ref{Theorem 5.2} and the last inequality, we have
\begin{align}\label{1new}
		\left(\dashint_{\Omega} g_{u_i}^{q_i+\delta_i}\,\textrm{d}\mu\right)^{\frac{1}{q_i+\delta_i}}&\leq C_i\left(\left(\dashint_{\Omega}g_{u_i}^{q_i}\textrm{d}\mu\right)^{\frac{1}{q_i}}+\left(\dashint_{\Omega}g_w^{q_i+\delta_i}\,\dd\mu\right)^{\frac{1}{q_i+\delta_i}}+1\right)\nonumber\\
		&\leq C_i\left(\left(\dashint_{\Omega}g_w^{q_i}\,\dd\mu\right)^{\frac{1}{q_i}}+\left(\dashint_{\Omega}g_w^{q_i+\delta_i}\,\dd\mu\right)^{\frac{1}{q_i+\delta_i}}+2\right),
	\end{align}
with $C_i$ depending on $p_i, q_i,\alpha,\beta$ and $K$. Notice that, by H\"older inequality, we get
$$\left(\dashint_{\Omega}g_w^{q_i}\,\textrm{d}\mu\right)^{\frac{1}{q_i}}\leq \left(\dashint_{\Omega}g_w^{q_i+\delta_i}\,\textrm{d}\mu\right)^{\frac{1}{q_i+\delta_i}}.$$
Therefore, 
\begin{align*}
\left(\dashint_{\Omega}g_{u_i}^{q_i+\delta_i}\,\textrm{d}\mu\right)^{\frac{1}{q_i+\delta_i}}&\leq C_i\left(2\left(\dashint_{\Omega}g_w^{q_i+\delta_i}\,\textrm{d}\mu\right)^{\frac{1}{q_i+\delta_i}}+2\right)\\
&\leq C_i\left(\left(\dashint_{\Omega}g_w^{q_i+\delta_i}\,\textrm{d}\mu\right)^{\frac{1}{q_i+\delta_i}}+1\right).
\end{align*}
By the last inequality we obtain
\begin{equation}\label{labuena}
\left(\dashint_{\Omega} g_{u_i}^{q_i+\delta_i}\,\dd\mu\right)^{\frac{1}{q_i+\delta_i}}\leq C_i\left(\left(	\dashint_{\Omega} g_w^{q_i+\delta_i} \,\dd\mu\right)^{\frac{1}{q_i+\delta_i}}+1\right).
\end{equation}
Since $\lim_{i\rightarrow\infty}q_i=q$, there exists $N\in\mathbb{N}$ such that for every $i\geq N$ we have $q+\epsilon_0\leq q_i+\delta_0\leq q_i+\delta_i\leq \bar{q}$, where $\epsilon_0=\frac{\delta_0}{2}$. By the uniform boundedness of $C_i$ given by Remark \ref{things}, the H\"older inequality and \eqref{labuena}, we get
\begin{align*}
\left(\dashint_{\Omega}g_{u_i}^{q+\epsilon_0}\,\dd\mu\right)^{\frac{1}{q+\epsilon_0}}&\leq C\left(\dashint_{\Omega}g_{u_i}^{q_i+\delta_i}\,\dd\mu\right)^{\frac{1}{q_i+\delta_i}}\\
&\leq C\left(\left(\dashint_{\Omega}g_{w}^{\bar{q}}\,\dd\mu\right)^{\frac{1}{\bar{q}}}+1\right)<\infty.
\end{align*}
As a consequence, we obtain
\begin{equation}\label{4.4}
\sup_i\Vert g_{u_i-w}\Vert_{L^{q+\epsilon_0}(\Omega)}<\infty.
\end{equation}
We now consider $B_0=B(x_0,r_0)\supset\Omega$. In $S=\lbrace x\in B_0: u(x)=w(x)\rbrace$ we have $g_{u_i-w}=0$ $\mu$-a.e., where $g_{u_i-w}$ is the minimal $q_i$-weak upper gradient of $u_i-w$. Moreover, $u_i-w=0$, $q_i$-q.e. on $X\setminus\Omega$ and therefore $\mu$-a.e. on $X\setminus\Omega$. By decreasing $\epsilon_0$ if necessary, we can further assume that $q+\epsilon_0<s^*$ and therefore, $X$ supports a $(q+\epsilon_0,s)$-Poincar\'e inequality. By H\"older inequality, we easily get that $X$ supports a $(q+\epsilon_0, q+\epsilon_0)$-Poincar\'e inequality as well, therefore we can use Proposition \ref{prop2.3} for $u_i-w\in N^{1,q}(X)$. We also note that $p$-fatness implies $q+\epsilon_0$-fatness because $q+\epsilon_0>p$. Thus, ${\rm cap}_{q+\epsilon_0}(S,2B_0)\,r^{q+\epsilon_0}\geq C\mu(B_0),$ and so
\begin{align*}
\left(\int_\Omega\vert u_i-w\vert^{q+\epsilon_0}\,\dd\mu\right)^{\frac{1}{q+\epsilon_0}}&\leq\left(\mu(2B_0)\dashint_{2B_0}\vert u_i-w\vert^{q+\epsilon_0}\,\dd\mu\right)^{\frac{1}{q+\epsilon_0}}\\
&\leq\left(\frac{C\mu(B_0)}{{\rm cap}_{q+\epsilon_0}(S,2B_0)}\int_{2\lambda B_0} g_{u_i-w}^{q+\epsilon_0}\,\dd\mu\right)^{\frac{1}{q+\epsilon_0}}\\
&\leq C r_0\left(\int_{B_0} g_{u_i-w}^{q+\epsilon_0}\,\dd\mu\right)^{\frac{1}{q+\epsilon_0}}\\
&=Cr_0\,\mu(\Omega)^{\frac{1}{q+\epsilon_0}}\left(\dashint_{\Omega}g_{u_i-w}^{q+\epsilon_0}\,\dd\mu\right)^{\frac{1}{q+\epsilon_0}}\\
&\leq Cr_0\,\mu(\Omega)^{\frac{1}{q+\epsilon_0}}\left(\dashint_{\Omega}(g_{u_i}+g_{w})^{q+\epsilon_0}\,\dd\mu\right)^{\frac{1}{q+\epsilon_0}}\\
&\leq C\,2^{1-\frac{1}{q+\epsilon_0}}r_0\,\mu(\Omega)^{\frac{1}{q+\epsilon_0}}\left(\dashint_{\Omega}g_{u_i}^{q+\epsilon_0}\,\dd\mu+   \dashint_{\Omega}g_{w}^{q+\epsilon_0}\,\dd\mu \right)^{\frac{1}{q+\epsilon_0}}\\
&\leq Cr_0\,\mu(\Omega)^{\frac{1}{q+\epsilon_0}}\left(\left(\dashint_{\Omega}g_{u_i}^{q+\epsilon_0}\,\dd\mu\right)^{\frac{1}{q+\epsilon_0}}+\left(\dashint_{\Omega}g_{w}^{q+\epsilon_0}\,\dd\mu\right)^{\frac{1}{q+\epsilon_0}}\right)\\
&\leq Cr_0\,\mu(\Omega)^{\frac{1}{q+\epsilon_0}}\left(\left(\dashint_{\Omega}g_w^{\bar{q}}\,\dd\mu\right)^\frac{1}{\bar{q}}+1\right)<\infty,
\end{align*}
where in the latter step we applied H\"older inequality. So, by this last inequality and \eqref{4.4}, we get $\sup_{i}\Vert u_i-w\Vert_{N^{1,q+\epsilon_0}(\Omega)}<\infty$.
Meaning, the sequence $(u_i-w)$ is uniformly bounded in $N^{1,q+\epsilon_0}(\Omega)$. By extending $u_i-w$ to zero in $X\setminus\Omega$, we obtain that the sequence
$(\Vert u_i-w\Vert_{L^{1}(B_0)}+\Vert g_{u_i-w}\Vert_{L^{1}(5\lambda B_0)})$
is bounded.\\
Theorem \ref{rellich} implies the existence of $\tilde{u}\in L^{q+\epsilon_0}(B_0)$ and a subsequence $(u_{i_k})$ such that $u_{i_k}-w\rightarrow\tilde{u}-w\ \ \textrm{in}\ L^{q+\epsilon_0}(B_0)$. As $u_i\rightarrow u$ $\mu$-a.e. in $\Omega$, then $\tilde{u}=u$ $\mu$-a.e. in $\Omega$ and $u\in L^{q+\epsilon_0}(\Omega)$. By the uniform boundedness  of $\Vert g_{u_{i_k}}\Vert_{L^{q+\epsilon_0}(\Omega)}$, we find $g\in L^{q+\epsilon_0}(\Omega)$ and a subsequence, still called $g_{u_{i_k}}$ to simplify notation, such that $g_{u_{i_k}}\rightarrow g\ \textrm{weakly in }L^{q+\epsilon_0}(\Omega).$ Finally, using Lemma 3.6 in \cite{Sh} we get that $g$ is a weak upper gradient of $u\in N^{1,q+\epsilon_0}(\Omega)$.
\end{proof}

Now, we consider a compact set $D\subset \Omega$ and define $D(t) = \{x \in \Omega:{\rm dist}(x,D) < t\}$, for every $t > 0$.
Then $\overline{D(t)} \subset \Omega$ for $t \in ]0, t_0[$, where $t_0 = {\rm dist}(D,X\setminus \Omega)$. We give a double phase version of Lemma 4.4 in \cite{MZG}, which was already adapted from a result proven by Kinnunen and Martio \cite{KM} concerning super-quasiminimizers. More specifically, we obtain a local uniform integrability estimate for the minimal upper gradients.
\begin{lemma}\label{Lemma4.4}
Let $u_i$, $u$ be as in Theorem \ref{Theorem 1.1.}. Then 
\begin{equation*}
\limsup_{i \to \infty}\int_{D(t)}(g_{u_i}^{p_i}+ g_{u_i}^{q_i})\,\dd\mu \leq C \int_{D(t)} (g_{u}^{p}+g_{u}^{q})\, \dd\mu,
\end{equation*}
for almost every $t \in ]0, t_0[$, with the constant $C$ depending on $\alpha, \beta$ and $K$.
\end{lemma}

\begin{proof} We define $\phi_i= \eta (u-u_i)$ where $\eta$ is a Lipschitz cut-off function given by
$$\eta =\begin{cases}
	1 \quad \mbox{on $D(t')$,}\\
	0 \quad \mbox{on $\Omega \setminus D(t)$,}
\end{cases}$$
with $0 < t' < t < t_0$. 
Since $\lim_{i\rightarrow\infty}q_i=q$, there exists $N\in\mathbb{N}$ such that for every $i\geq N$ we have $p_i<q_i<q+\epsilon_0$. Notice that $\phi_i\in N^{1,q_i}_0(D(t))$, because $u_i$ and $u$ belong to 
$N^{1,q+\epsilon_0}(\Omega)$. As a consequence of the quasiminimizing property \eqref{min} of $u_i$, we obtain
\begin{equation*}
\int_{D(t')}(a g_{u_i}^{p_i}+b g_{u_i}^{q_i})\,\dd\mu \leq \int_{D(t)}(a g_{u_i}^{p_i}+b g_{u_i}^{q_i})\,\dd\mu\leq K \int_{D(t)}(a g_{u_i+\phi_i}^{p_i}+b g_{u_i+\phi_i}^{q_i})\,\dd\mu.
\end{equation*}
Now, using Lemma 2.1 of \cite{MZG}, we get
$g_{u_i+\phi_i}\leq (1-\eta)g_{u_i}+g_{\eta}|u-u_i|+\eta g_u,$ therefore
\begin{align*}
	\int_{D(t')}(a g_{u_i}^{p_i}+b g_{u_i}^{q_i})\,\dd\mu&\leq C\Bigg(\int_{D(t)}a(1-\eta)^{p_i}g_{u_i}^{p_i}\,\dd\mu+\int_{D(t)}ag_{\eta}^{p_i}|u-u_i|^{p_i} \,\dd\mu + \int_{D(t)} a\eta^{p_i}g_u^{p_i}\,\dd\mu\\&\quad\quad+ \int_{D(t)}b(1-\eta)^{q_i}g_{u_i}^{q_i}\,\dd\mu +\int_{D(t)}bg_{\eta}^{q_i}|u-u_i|^{q_i}\,\dd\mu + \int_{D(t)} b\eta^{q_i}g_u^{q_i} \,\dd\mu\Bigg)
	\\ &\leq C\Bigg(\int_{D(t)}(1-\eta)^{p_i}g_{u_i}^{p_i}\,\dd\mu +\int_{D(t)}g_{\eta}^{p_i}|u-u_i|^{p_i}\,\dd\mu + \int_{D(t)} \eta^{p_i}g_u^{p_i} \,\dd\mu\\&\quad\quad+ \int_{D(t)}(1-\eta)^{q_i}g_{u_i}^{q_i}\,\dd\mu +\int_{D(t)}g_{\eta}^{q_i}|u-u_i|^{q_i} \,\dd\mu + \int_{D(t)} \eta^{q_i}g_u^{q_i} \,\dd\mu\Bigg),
\end{align*}
with $C$ depending on $\beta$ and $K$. Thus,
\begin{align*}
	\int_{D(t')}(g_{u_i}^{p_i}+g_{u_i}^{q_i})\,\dd\mu	&\leq C\Bigg(\int_{D(t)}(1-\eta)^{p_i}g_{u_i}^{p_i}\,\dd\mu +\int_{D(t)}g_{\eta}^{p_i}|u-u_i|^{p_i} \,\dd\mu + \int_{D(t)} \eta^{p_i}g_u^{p_i}\,\dd\mu\\&\quad\quad+ \int_{D(t)}(1-\eta)^{q_i}g_{u_i}^{q_i}\,\dd\mu +\int_{D(t)}g_{\eta}^{q_i}|u-u_i|^{q_i}\,\dd\mu + \int_{D(t)} \eta^{q_i}g_u^{q_i}\,\dd\mu\Bigg),
\end{align*} 
where $C$ also depends on $\alpha$.

\noindent By definition $\eta=1$ on $D(t')$, therefore by adding $C\int_{D(t')}(g_{u_i}^{p_i}+g_{u_i}^{q_i})\,\dd\mu$ to both sides of the inequality, we get
\begin{align}\label{newone}
(1+C)\int_{D(t')}& (g_{u_i}^{p_i}+g_{u_i}^{q_i})\,\dd\mu\nonumber\\ 
	&\leq C\Bigg(\int_{D(t)}g_{u_i}^{p_i}\,\dd\mu +\int_{D(t)}g_{\eta}^{p_i}|u-u_i|^{p_i} \,\dd\mu + \int_{D(t)} \eta^{p_i}g_u^{p_i} \,\dd\mu\nonumber\\&\quad\quad+ \int_{D(t)}g_{u_i}^{q_i}\,\dd\mu +\int_{D(t)}g_{\eta}^{q_i}|u-u_i|^{q_i}\,\dd\mu + \int_{D(t)} \eta^{q_i}g_u^{q_i} \,\dd\mu\Bigg).
\end{align} 
We consider a nondecreasing function $\psi$ defined on the interval $]0, t_0[$ as
\begin{equation*}
	\psi(t)= \limsup_{i \to \infty}\int_{D(t)}(g_{u_i}^{p_i}+g_{u_i}^{q_i})\,\dd\mu<\infty,
\end{equation*}
where the latter is given by the uniform higher integrability
of $u_i$. As a consequence, $\psi$ is discontinuous in at most a countable number of points. Now, we consider a point $t$ where $\psi$ is continuous. Passing to superior limit on both sides of \eqref{newone}, it follows
\begin{align*}
(1+C)\psi(t')\leq& C \psi(t)+C \int_{D(t)} (g_{u}^{p}+ g_{u}^{q})\,\dd\mu\\&+ C \limsup_{i \to \infty}\int_{D(t)} |u-u_i|^{p_i}\,\dd\mu+ C \limsup_{i \to \infty}\int_{D(t)} |u-u_i|^{q_i}\,\dd\mu .
\end{align*}
Using H\"{o}lder inequality and Lemma \ref{Lemma4.2}, we deduce that
$$
\limsup_{i \to \infty}\int_{D(t)} |u-u_i|^{p_i}\,\dd\mu=\limsup_{i \to \infty}\int_{D(t)} |u-u_i|^{q_i}\,\dd\mu= 0.
$$
Recalling that $\psi$ is continuous at $t$, we have
\begin{align*}
	(1+C)\psi(t)\leq C \psi(t)+C \int_{D(t)} (g_{u}^{p}+ g_{u}^{q})\,\dd\mu,
\end{align*}
that is
\begin{equation*}
\psi(t)\leq C\int_{D(t)} (g_{u}^{p}+g_{u}^{q})\, \dd\mu.
\end{equation*}
This concludes the proof.
\end{proof}

The next lemma will be needed as a first step in order to show that $u$ is a quasiminimizer of the $(p,q)$-energy integral with boundary data $w$. 
\begin{lemma}\label{Lemma12.13JKbook}
	Under the assumptions in Theorem \ref{Theorem 1.1.} and notation as in Lemma \ref{Lemma4.2}, we have $u-w \in N^{1,q}_0(\Omega)$.
\end{lemma}

\begin{proof}
Let $0<\epsilon<q-p$ and let $(u_i)$ be the subsequence given by Lemma \ref{Lemma4.2}.
For sufficiently large $i\in \mathbb{N}$, we have  $q_i>q-\epsilon$  and $u_i-w\in N^{1,q_i}_0(\Omega)$. Thus, $u_i-w\in N^{1,q-\epsilon}_0(\Omega)$ for every such $i\in \mathbb{N}$ and, eventually passing to a subsequence, we may assume that this holds for every $i\in \mathbb{N}$. 
As a consequence of the Sobolev inequality (see Remark \ref{Lemma 2.4EMThesis}), we deduce
\begin{equation*}
	\|u_i-w\|_{N^{1,q-\epsilon}_0(\Omega)} \leq C \|g_{u_i-w}\|_{L^{q-\epsilon}(\Omega)}\leq C \|g_{u_i-w}\|_{L^{q}(\Omega)}.
\end{equation*}
As a consequence, the norms of $u_i-w$ are uniformly bounded in $N^{1,q-\epsilon}_0(\Omega)$.
By the self-improving property of the $p$-fatness, we get that $X\setminus \Omega$ is uniformly $(q-\epsilon)$-fat for $\epsilon>0$ such that $p<q-\epsilon.$
Using Lemma \ref{compactnessr} and the fact that $u_i \to u$ $\mu$-a.e., we deduce that $u-w \in N^{1,q-\epsilon}_0(\Omega),$ for all $\epsilon>0$ such that $p < q-\epsilon$.
We also note that $p$-fatness implies $q$-fatness, since $p<q$. The proof is complete by applying Proposition \ref{iintersection} and the $q$-fatness of $X\setminus \Omega$.
\end{proof}

At this point, we are left to show that $u$  is a quasiminimizer of the $(p,q)$-energy integral with boundary data $w$. In order to do that, we begin with the following inequality, that is a lower semicontinuity result in the varying exponent case.
\begin{lemma}\label{Lemma12.15JKbook}
Under the assumptions in Theorem \ref{Theorem 1.1.} and notation as in Lemma \ref{Lemma4.2}, we have
		\begin{equation}\label{4.6}
			\int_E g_u^{q} \,\dd\mu \leq \liminf_{i \to \infty} \int_E g_{u_i}^{q_i} \,\dd\mu,
		\end{equation}
for every $\mu$-measurable subset $E$ of $\Omega$. 
\end{lemma}

\begin{proof} Since $\lim_{i\rightarrow\infty}q_i=q$, for every $\epsilon> 0$, there exists $N\in\mathbb{N}$ such that for every $i\geq N$ we have $1<q_i<q-\epsilon_0$. 
As a consequence of Lemma \ref{Lemma4.2}, $(g_{u_i})$ converges weakly to a weak upper gradient $g$ of $u$, therefore, for every $\mu$-measurable
subset $E$ of $\Omega$, we get
\begin{align*}
	\int_E g_u^{q-\epsilon} \,\dd\mu&\leq\int_E g^{q-\epsilon} \,\dd\mu \leq \liminf_{i \to \infty}\int_E g_{u_i}^{q-\epsilon} \,\dd\mu\\& \leq \liminf_{i \to \infty}\left(\int_E g_{u_i}^{q_i} \,\dd\mu\right)^{\frac{q-\epsilon}{q_i}} \mu(E)^{1-\frac{q-\epsilon}{q_i}}\\& \leq \liminf_{i \to \infty}\left(\int_E g_{u_i}^{q_i} \,\dd\mu\right)^{\frac{q-\epsilon}{q}} \mu(E)^{\frac{\epsilon}{q}}.
\end{align*}
Taking the limit as $\epsilon\rightarrow 0$, we conclude the proof.
\end{proof}
\begin{remark}\label{alsoforp}
We note that Lemma \ref{Lemma12.15JKbook} also holds for $p$ instead of $q$, meaning that, under the assumptions of Theorem \ref{Theorem 1.1.}, we get
\begin{equation}
			\int_E g_u^{p} \,\dd\mu \leq \liminf_{i \to \infty} \int_E g_{u_i}^{p_i} \,\dd\mu,
		\end{equation}
for every $\mu$-measurable subset $E$ of $\Omega$
\end{remark}

Finally, we are ready to show that the function $u$  is a quasiminimizer of the $(p,q)$-energy integral with boundary data $w$.

\begin{lemma}\label{Lemma12.18JKbook}
	Under the assumptions in Theorem \ref{Theorem 1.1.} and notation as in Lemma \ref{Lemma4.2}, we have 
\begin{equation}\label{4.5}
	\int_{\Omega'} (a g_{u}^p+bg_{u}^q)\, \dd \mu
	\leq C\int_{\Omega'}(a g_{u+\phi}^p+b g_{u+\phi}^q) \,\dd \mu.
\end{equation} 
for every bounded open subset $\Omega'$ of $\Omega$ with $\Omega'\Subset \Omega$ and for all functions $\phi \in N^{1,q}_0({\Omega'})$, where $C$ depends on $\alpha, \beta$ and $K$.
\end{lemma}

\begin{proof}
Firstly, we prove that inequality \eqref{4.5} holds for every $\phi \in {\rm Lip_C(\Omega')}$, where $\textrm{Lip}_C(\Omega')$ denotes the set of Lipschitz functions with compact support in $\Omega'$. Then, by an approximation approach, we deduce that inequality \eqref{4.5} is satisfied for all $\phi \in N^{1,q}_0(\Omega')$. Indeed, for every $\epsilon>0$ there exists $\phi_{\epsilon} \in {\rm Lip_C}(\Omega')$ such that $\|\phi_{\epsilon}-\phi\|_{ N^{1,q}(\Omega')}<\epsilon$, see Theorem 5.46 in \cite{BB}.

\noindent Let $\epsilon>0$ and $\Omega''$ and $\Omega_0$ open sets with $\Omega' \Subset \Omega'' \Subset \Omega_0 \Subset \Omega$ and 
\begin{equation*}
	\int_{\Omega_0\setminus \overline{\Omega'}} (g_u^p+g_u^{q}) \,\dd\mu < \epsilon.
\end{equation*}
 We consider a function $\phi_i$ defined as $\phi_i=\phi+\eta(u-u_i),$ where $\eta$ is a Lipschitz cut-off function defined as $\eta=1$ in a neighbourhood of $\overline{\Omega'}$ and $\eta=0$ in $\Omega \setminus \Omega''$. Notice that, for $i$ large enough,  $\phi_i \in N^{1,q_i}_0(\Omega'')$ because $\phi \in {\rm Lip_C(\Omega')}$ and $u_i, u \in N^{1,q+\epsilon_0}(\Omega)$. Using the quasiminimizing property \eqref{min} of $u_i$, with $u_i+\phi_i$ as test function, we have
\begin{align}\label{4.7}
	\int_{\Omega''} (ag_{u_i}^{p_i}+b g_{u_i}^{q_i}) \,\dd\mu &\leq K \int_{\Omega''} (a g_{u_i+\phi_i}^{p_i} +bg_{u_i+\phi_i}^{q_i}) \,\dd\mu \nonumber \\&\leq K \int_{\overline{\Omega'}} (ag_{u_i+\phi_i}^{p_i}+ bg_{u_i+\phi_i}^{q_i}) \,\dd\mu\nonumber \\ & \quad +K \int_{\Omega'' \setminus \overline{\Omega'}} (ag_{u_i+\phi_i}^{p_i}+bg_{u_i+\phi_i}^{q_i}) \,\dd\mu.
\end{align}
By the definition of $\eta$ in a neighborhood of $\overline{\Omega'}$, we obtain
\begin{equation}\label{4.8}
	u_i+\phi_i=u+\phi.
\end{equation}
Moreover,  $\phi=0$ in $\Omega'' \setminus \overline{\Omega'}$, thus $u_i+\phi_i=u_i+\eta (u-u_i)$ in $\Omega'' \setminus \overline{\Omega'}$. By applying Lemma 2.1 of \cite{MZG}, we get that $g_{u_i+\phi_i}\leq(1-\eta)g_{u_i}+\eta g_u+ g_{\eta}|u-u_i|$.
Then, we have
\begin{align}\label{4.9}
	\int_{\Omega'' \setminus \overline{\Omega'}} (ag_{u_i+\phi_i}^{p_i}+ bg_{u_i+\phi_i}^{q_i}) \,\dd\mu &\leq C  \int_{\Omega'' \setminus \overline{\Omega'}}\left((1-\eta)^{p_i}g_{u_i}^{p_i}+(1-\eta)^{q_i}g_{u_i}^{q_i}\right)\,\dd\mu\nonumber \\&\quad+ C  \int_{\Omega'' \setminus \overline{\Omega'}}\left(\eta^{p_i} g_u^{p_i}+\eta^{q_i} g_u^{q_i}\right) \,\dd\mu\nonumber\\&\quad+C  \int_{\Omega'' \setminus \overline{\Omega'}} \left(g_{\eta}^{p_i} |u-u_i|^{p_i}+g_{\eta}^{q_i} |u-u_i|^{q_i}\right)  \,\dd\mu,
\end{align}
where $C$ depends on $\beta$. 
We now consider the first term on the right–hand side of \eqref{4.9}. By the definition of $\eta$, we can choose a compact set $D \subset \overline{\Omega''}$ with $D \cap \overline{\Omega'}= \emptyset$ such that
\begin{equation*}
	\int_{\Omega'' \setminus \overline{\Omega'}}\left((1-\eta)^{p_i}g_{u_i}^{p_i}+(1-\eta)^{q_i}g_{u_i}^{q_i}\right)\,\dd\mu \leq \int_D (g_{u_i}^{p_i}+g_{u_i}^{q_i})\,\dd\mu.
\end{equation*}
We choose $t$ such that $D(t)\subset \Omega_0 \setminus \overline{\Omega'}$ and for which we can use Lemma \ref{Lemma4.4} to get
\begin{equation*}
	\limsup_{i \to \infty} \int_{D(t)}
	(g_{u_i}^{p_i} + g_{u_i}^{q_i}) \,\dd\mu\leq C \int_{D(t)} (g_u^p+ g_u^q)\,\dd\mu,
\end{equation*}
where $C$ depends also on $\alpha$ and $K$. By the choice of $\Omega_0$, this implies
\begin{align}\label{4.10}
	\limsup_{i \to \infty}\int_{\Omega'' \setminus \overline{\Omega'}}&\left((1-\eta)^{p_i}g_{u_i}^{p_i}+(1-\eta)^{q_i}g_{u_i}^{q_i}\right)\,\dd\mu\nonumber\\&\leq\limsup_{i \to \infty}\int_D (g_{u_i}^{p_i}+ g_{u_i}^{q_i})\,\dd\mu\nonumber\\ &\leq\limsup_{i \to \infty}\int_{D(t)} (g_{u_i}^{p_i}+ g_{u_i}^{q_i})\,\dd\mu\nonumber\\ & \leq C \int_{D(t)} (g_u^p+ g_u^q)\,\dd\mu \leq C \epsilon.
\end{align}
Now, we consider the second and third integrals on the right-hand side of \eqref{4.9}. Analogously, because of the definition of $\Omega_0$, we get
\begin{align}\label{4.11}
	\limsup_{i \to \infty} \int_{\Omega'' \setminus \overline{\Omega'}}(\eta^{p_i} g_u^{p_i}+\eta^{q_i} g_u^{q_i}) \,\dd\mu&\leq \int_{\Omega'' \setminus \overline{\Omega'}}(g_u^p+g_u^q) \,\dd\mu \nonumber \\ &\leq \int_{\Omega_0 \setminus \overline{\Omega'}} (g_u^p+ g_u^q) \,\dd\mu \leq \epsilon.
\end{align}
We recall that the minimal $q$-weak upper gradient of a Lipschitz function is bounded by its Lipschitz constant $\mu$-a.e., thus, by H\"older inequality and Lemma \ref{Lemma4.2}, we have
\begin{align}\label{4.12}
	\limsup_{i \to \infty}&\int_{\Omega'' \setminus \overline{\Omega'}} \left(g_{\eta}^{p_i} |u-u_i|^{p_i} +g_{\eta}^{q_i} |u-u_i|^{q_i}\right)  \,\dd\mu\nonumber \\ &\leq C \limsup_{i \to \infty}\int_{\Omega'' \setminus \overline{\Omega'}}  \left(|u-u_i|^{p_i} +|u-u_i|^{q_i}\right)  \,\dd\mu\nonumber \\ &\leq C \limsup_{i \to \infty} \left( \int_{\Omega'' \setminus \overline{\Omega'}}  |u-u_i|^{q+\epsilon_0}  \,\dd\mu\right)^{\frac{p_i}{q+\epsilon_0}}\nonumber \\ &\quad+C \limsup_{i \to \infty} \left( \int_{\Omega'' \setminus \overline{\Omega'}}  |u-u_i|^{q+\epsilon_0}  \,\dd\mu\right)^{\frac{q_i}{q+\epsilon_0}}=0.
\end{align}  
From \eqref{4.9}, together with the estimates \eqref{4.10}, \eqref{4.11} and \eqref{4.12}, we obtain 
\begin{equation}\label{4.13}
	\limsup_{i \to \infty} \int_{\Omega'' \setminus \overline{\Omega'}}\left( ag_{u_i+\phi_i}^{p_i}+ bg_{u_i+\phi_i}^{q_i}\right)\,\dd\mu\leq C\epsilon.
\end{equation}
Now, using \eqref{4.6}, Remark \ref{alsoforp}, \eqref{4.7}, \eqref{4.8} and \eqref{4.13}, we deduce
\begin{align}\label{4.14}
	\int_{\overline{\Omega'}}(ag_u^p+bg_u^q) \,\dd\mu &\leq \liminf_{i \to \infty} \int_{\Omega''}(ag_{u_i}^{p_i}+bg_{u_i}^{q_i})\,\dd\mu\nonumber\\&
	\leq	 C\liminf_{i \to \infty} \int_{\Omega''}(ag_{u_i+\phi_i}^{p_i}+bg_{u_i+\phi_i}^{q_i})\,\dd\mu\nonumber\\&
	\leq C\liminf_{i \to \infty}  \int_{\overline{\Omega'}} ag_{u+\phi}^{p_i}\,\dd\mu + C \liminf_{i \to \infty}  \int_{\Omega'' \setminus \overline{\Omega'}} ag_{u_i+\phi_i}^{p_i}\,\dd\mu\nonumber\\&\quad+C\liminf_{i \to \infty}  \int_{\overline{\Omega'}} bg_{u+\phi}^{q_i}\,\dd\mu  + C \liminf_{i \to \infty}  \int_{\Omega'' \setminus \overline{\Omega'}} bg_{u_i+\phi_i}^{q_i}\,\dd\mu\nonumber\\&\leq C \int_{\overline{\Omega'}} (ag_{u+\phi}^{p}+bg_{u+\phi}^{q})\,\dd\mu + C\epsilon,
\end{align}
Taking the limit as $\epsilon\rightarrow 0$, we conclude the proof for every $\phi \in {\rm Lip_C}(\Omega')$ and thus, by approximation, for every $\phi\in N^{1,q}_0(\Omega')$.
\end{proof}

Now, the proof of the stability result stated in Theorem \ref{Theorem 1.1.} follows immediately from the previous lemmata.


\begin{thebibliography}{10}
\bibitem{BDKS} V. B\"{o}gelein, F. Duzaar, J. Kinnunen,  C. Scheven, \textit{Higher integrability for doubly nonlinear parabolic
systems}, J. Math. Pures Appl., 143 (2020), 31--72.

\bibitem{Boj} B. V. Bojarski, \textit{Generalized solutions of a system of differential equations of first order and elliptic type with discontinuous coefficients} (Russian), Mat. Sb. N. S., 43 (1957), 451--503.

	 \bibitem{B} J. Bj\"{o}rn, \textit{Boundary continuity for quasiminimizers on metric spaces}, Illinois J. Math., 46 (2002), 383--403.
	 
	 
    \bibitem{BB}  A. Bj\"{o}rn, J. Bj\"{o}rn, \textit{Nonlinear potential theory on metric spaces}, EMS Tracts in Mathematics, 17, European Mathematical Society (EMS), Zurich (2011). 
	


	\bibitem{BMS} J.  Bj\"{o}rn, P. MacManus, N. Shanmugalingam,  \textit{Fat sets and pointwise boundary estimates for $p$-harmonic functions in metric spaces}, J. Anal. Math., 85 (2001), 339--369.
\bibitem{CW} Y.Z. Chen, L.C. Wu,  \textit{Second order elliptic equations and elliptic systems}, volume 174 of Translations
of Mathematical Monographs. American Mathematical Society, Providence, RI, (1998). Translated from the 1991 Chinese original by Bei Hu.

\bibitem{CM} M. Colombo, G. Mingione, \textit{Regularity for Double Phase Variational Problems}, Arch. Rational. Mech. Anal., 215 (2015), 443--496. 
	
	
\bibitem{DGV} E. DiBenedetto, U. Gianazza, V. Vespri, \textit{Remarks on local boundedness and local H\"older continuity of local weak solutions to anisotropic $p$-Laplacian type equations}, J. Elliptic Parabol. Equ., 2 (2016), 157--169.
	
\bibitem{EM} A. Elcrat, N. Meyers, \textit{Some results on regularity for solutions of non-linear elliptic
systems and quasi-regular functions}, Duke Math. J., 42 (1975), 121-–136.
	
\bibitem{ELM} L. Esposito, F. Leonetti, G. Mingione, \textit{Higher integrability for minimizers of integral functionals with (p, q) growth}, J. Diff. Equ.,  157 (1999), 414--438. 

\bibitem{FH} Y. Fujishima, J. Habermann, \textit{Global higher integrability for non-quadratic parabolic quasiminimizers
on metric measure spaces}, Adv. Calc. Var., 10 (2017), 267--301.

\bibitem{GS}U. Gianazza, S. Schwarzacher, \textit{Self-improving property of degenerate parabolic equations of porous
medium-type}, Amer. J. Math., 141 (2019), 399--446.

\bibitem{Gia} M. Giaquinta, \textit{Growth conditions and regularity, a counterexample}, Manuscripta Math., 59 (1987), 245--248. 

\bibitem{GG1} M. Giaquinta, E. Giusti, \textit{On the regularity of the minima of variational integrals}, Acta Math., 148 (1982), 31--46. 


\bibitem{GG2} M. Giaquinta, E. Giusti, \textit{Quasi-minima}, Ann. Inst. H. Poincar\'e Anal. Non Lin\'eaire, 1 (1984), 79--107.

\bibitem{GM} M. Giaquinta, L. Martinazzi, \textit{An introduction to the regularity theory for elliptic systems, harmonic
maps and minimal graphs}, volume 11 of Appunti. Scuola Normale Superiore di Pisa (Nuova Serie) [Lecture Notes. Scuola Normale Superiore di Pisa (New Series)]. Edizioni della Normale, Pisa, second edition, 2012.
    
    \bibitem{G} E. Giusti, \textit{Direct Methods in the Calculus of Variations}, World Scientific Publishing, River
    Edge, NJ, (2003).
    
    
    \bibitem{Gra} S. Grandlund, \textit{An $L^p$-estimate for the gradient of extremals}, Math. Scand, 50 (1982), 66--72.
    
    
 

    
\bibitem{HedKil} L. I. Hedberg, T. Kilpel\"ainen, \textit{On the stability of Sobolev spaces with zero boundary
values}. Math\`a Scand\`a, 85 (1999), 245–-258.    
    

\bibitem{KZ}	S. Keith, X. Zhong, \textit{The Poincar\'{e} inequality is an open ended condition}, Ann. of Math., 167 (2008), 575--599. 
	
	
\bibitem{KKM} T.  Kilpel\"ainen,  J.  Kinnunen,  O.  Martio, \textit{Sobolev  Spaces  with  Zero  Boundary  Values  on Metric Spaces}, Potential Anal., 12 (2000), 233–-247.


\bibitem{KilKos} T. Kilpel\"anen, P. Koskela, \textit{Global integrability of the gradients of solutions to partial differential equations}, Nonlinear Anal., 23 (1994), 899--909.
	

\bibitem{KLM} J. Kinnunen, J. Lehrb\"{a}ck, A. V\"{a}h\"{a}kangas,	\textit{Maximal Function Methods for Sobolev Spaces}, Mathematical Surveys and Monographs, 257, American Mathematical Soc., (2021).


\bibitem{KMM} J.  Kinnunen, N. Marola, O.  Martio, \textit{Harnack's principle for quasiminimizers}, 	Ric. Mat., 56 (2007), 73--88.	

	
\bibitem{KM} J. Kinnunen, O. Martio, \textit{Potential theory of quasiminimizers}, Ann. Acad. Sci. Fenn. Math., 28 (2003), 459--490.


\bibitem{KP} J. Kinnunen,  M. Parviainen, \textit{Stability for degenerate parabolic equations}, 3 (2010), 29--48.


\bibitem{KS} J. Kinnunen, N. Shanmugalingam, \textit{Regularity of quasi-minimizers on metric spaces}, Manuscripta Math., 105 (2001), 401--423.


\bibitem{LM1} G. Li, O. Martio, \textit{Stability and higher integrability of derivatives of solutions in double obstacle problem}, J. Math. Anal. Appl., 272  (2002), 19--29.


\bibitem{LM} G. Li, O. Martio, \textit{Stability of solutions of varying degenerate elliptic equations},
Indiana Univ. Math. J., 47  (1998), 873--891.


\bibitem{L2} P. Lindqvist, \textit{Stability for the solutions of ${\rm div} (|\nabla u|^{p-2} \nabla u) = f$ with varying $p$}, J. Math. Anal. Appl., 127  (1987), 93--102.
	
	
\bibitem{Maa} O. E. Maasalo, \textit{The Gehring lemma in metric spaces}, arXiv: 0704.3916v3, (2008).	

	
\bibitem{MZG} O. E. Maasalo, A. Zatorska-Goldstein, \textit{Stability of quasiminimizers of the $p$-Dirichlet integral with varying $p$ on metric spaces}, J. London Math. Soc., 77  (2008), 771--788.


\bibitem{Ma0} P. Marcellini, \textit{Anisotropic and $p, q$-nonlinear partial differential equations}, Rend. Lincei Sci. Fis. Nat., 31 (2020), 295–301.	


\bibitem{Ma1} P. Marcellini, \textit{Regularity and existence of solutions of elliptic equations with $p$, $q$-growth conditions}, J. Differential Equations, 90 (1991), 1--30.


\bibitem{Ma2} P. Marcellini, \textit{Regularity for elliptic equations with general growth conditions}, J. Differential Equations, 105 (1993), 296--333.


\bibitem{Ma3}P. Marcellini, \textit{Regularity of minimizers of integrals of the calculus of variations with non standard growth conditions}, Arch. Rational Mech. Anal., 105 (1989), 267--284.


\bibitem{MMM} M. Masson, M. Miranda Jr., F. Paronetto, M. Parviainen, \textit{Local higher integrability
for parabolic quasiminimizers in metric spaces}, Ric. Mat., 62 (2013), 279–-305.


\bibitem{MM} M. Masson, M. Parviainen, \textit{Global higher integrability for parabolic quasiminimizers in
metric measure spaces}, J. Anal. Math., 126 (2015), 307–-339.


\bibitem{MSSS} K. Moring, C. Scheven, S. Schwarzacher, T. Singer, \textit{Global higher integrability of weak solutions of
porous medium systems}, Commun. Pure Appl. Anal., 19 (2020), 1697--1745.


\bibitem{NP} A. Nastasi, C. Pacchiano Camacho, \textit{Regularity properties for quasiminimizers of a $(p,q)$-Dirichlet integral}, Calc. Var., 227  (2021), 1--37.


\bibitem{Sh} N. Shanmugalingam, \textit{Harmonic functions on metric spaces}, Illinois J. Math., 45  (2001), 1021--1050.
    

\bibitem{ZG} A. Zatorska-Goldstein, \textit{Very weak solutions of nonlinear subelliptic equations}. Ann. Acad. Sci. Fenn. Math., 30 (2005), 407--436.


\bibitem{Z}V. V. Zhikov, \textit{On some variational problems}, Russian J. Math. Phys., 5 (1997), 105--116.
     
 
\end{thebibliography}
\end{document}